\documentclass[letter,12pt]{amsart}
\usepackage{wrapfig}
\usepackage{amsfonts}
\usepackage{amssymb}
\usepackage{amsmath,tabu}
\usepackage{amsthm}
\usepackage{comment}
\usepackage{graphics}
\usepackage{epstopdf}
\usepackage{pst-all,psfrag}
\usepackage[unicode]{hyperref}
\usepackage{mathtools}
\usepackage{subcaption}
\usepackage{multicol}

\usepackage{verbatim}
\usepackage{eurosym}
\usepackage[utf8]{inputenc}
\usepackage[english]{babel}
\usepackage[margin=1in]{geometry}
\usepackage{apptools}
\usepackage{tabularx}
\usepackage{enumitem}
\usepackage{bbm}

\newtheorem{theorem}{Theorem}
\newtheorem{proposition}[theorem]{Proposition}

\newtheorem{lemma}[theorem]{Lemma}
\newtheorem{conjecture}[theorem]{Conjecture}
\newtheorem{question}[theorem]{Question}

\newtheorem{corollary}[theorem]{Corollary}
\newtheorem{remark}[theorem]{Remark}

\newtheorem{definition}[theorem]{Definition}
\newtheorem{example}[theorem]{Example}

\numberwithin{theorem}{section}

\newcommand{\eps}{\varepsilon}

\newcommand{\inv}{\mathrm{inv}}

\newcommand{\be}{\begin{equation}}
\newcommand{\ee}{\end{equation}}

\renewcommand{\S}{{\mathcal S}}
\newcommand{\old}[1]{}

\newcommand{\ii}{\mathbf i}
\newcommand{\dd}{\mathrm d}

\title{Six-vertex model with rare corners and random restricted permutations}

\author{Vadim Gorin \and Richard Kenyon}

\address{Vadim Gorin\\
	Departments of Statistics and Mathematics, UC Berkeley}
\email{vadicgor@gmail.com}
\address{Richard Kenyon\\
	Department of Mathematics, Yale University, New Haven, 06920}

\email{richard dot kenyon at yale.edu}

\begin{document}
\date{\today}
\maketitle

\begin{abstract}
We study limit shapes in two equivalent models: the six-vertex model in the $c\to0$ limit and the random
Mallows permutation with restricted permutation matrix. We give the Euler-Lagrange equation for the limit
shape and show how to solve it for a class of rectilinear polygonal domains. Its solutions
are given by piecewise-algebraic functions with lines of discontinuities.
\end{abstract}

\tableofcontents

\section{Introduction}

The six-vertex model (see Figures \ref{Figure_six_vertices}, \ref{Figure_DWBC}) is one of the most fundamental and well-studied models of statistical mechanics, with deep connections
to combinatorics, probability, and quantum algebra, we refer to \cite{Lieb_ferroelectric_models,baxter2007exactly,reshetikhin2010lectures,Bleher-Liechty14,Gorin_Nicoletti_lectures} for extensive reviews. Despite spectacular progress on exact solutions by Lieb \cite{lieb1967exact}, Sutherland-Yang-Yang \cite{sutherland1967exact},
Nolden \cite{nolden1992asymmetric}, Noh-Kim \cite{NohKim}, and others, in the most general setting it remains unsolved:
even closed expressions for the partition function on the torus or free energy as a function of the slope are not known.

In the combinatorics community, a six-vertex model setting of particular interest is given by ``domain wall boundary conditions'' (DWBC), as in Figure \ref{Figure_DWBC}. This setup has a number of connections with algebraic and integrable combinatorics:
special cases enumerate Alternating Sign Matrices, domino tilings of the Aztec Diamond, classes of symmetric polynomials, and so on, see e.g.\ \cite{elkies1992alternating,bressoud1999alternating,gier2009fully,zinn2009six,di2018integrable,zinn2024integrability} and references therein.

In this paper we study the $c\to0$ limit of the six-vertex model. For DWBC this is a point where the model
is isomorphic to a random Mallows permutation, that is a permutation $\sigma$ weighted by $q^{\inv(\sigma)}$,
where $\inv(\sigma)$  is the number of inversions (and $q=b^2/a^2$ is a positive parameter). The Mallows measure has been intensively studied since its introduction in \cite{mallows1957non}; we refer to \cite{he2022cycles,adamczak2024global} for two recent papers and many references to previous work. The closest to our text are papers \cite{starr2009thermodynamic,starr2018phase} which study the macroscopic limit shapes for random Mallows permutations.

In the present work we go beyond DWBC for the six-vertex model and obtain
explicit limit shapes for a variety of natural domains, including polygonal subdomains for the domain-wall boundary conditions. In the random Mallows permutations language, we deal with permutations with restricted positions, as in \cite{diaconis2001statistical}. Our limit shapes satisfy a variational principle whose associated PDE is the
hyperbolic Liouville equation
$$(\log g)_{xy} = 2r g,$$
for a function $g=g(x,y)$, where $r$ is a constant governing the speed of convergence of $q$ to $1$ as the domain grows,
see Proposition \ref{Proposition_EL} and Remark \ref{Remark_Liouville}. Surprisingly, our solutions are usually only piecewise analytic and in fact discontinuous, see Figures \ref{Figure_44simulations},\ref{Figure_22simulations},\ref{Figure_33simulations},\ref{Figure_32_nonconvex},\ref{Figure_triangle}. We are in a remarkable situation
where we can find explicit equations for \emph{nonanalytic} limit shapes.

In the $q=1$ case we prove that the variational problem has a unique maximizer, that is, limit shape, for any rectilinear domain, and
give a robust procedure for computing it, based on
the Sinkhorn-Knopp algorithm. For general $q\ne1$ (equivalently, $r\ne 0$) we do not know whether the Euler-Lagrange equations associated to the variational problem have a unique solution. However, for a wide class of cases with $q$ close to $1$, we are able to prove the uniqueness and further express the limit shape in terms of solutions to certain algebraic equations, see Theorems \ref{Theorem_maximizer_small_r} and \ref{Theorem_convex_solution}. As an outcome, in this situation the function $g(x,y)$ encoding the limit shape is a piecewise-rational function of $e^{rx}$, $e^{ry}$, see Corollary \ref{Corollary_g_form} for the exact statement. We supplement our general theorems with explicit computations for a number of examples in Sections \ref{Section_smooth} and \ref{Section_examples}.

We treat our results as an important contribution to the long-standing open problem of finding the limit shapes for the six-vertex model, discussed, e.g., in \cite{zinn2002influence,palamarchuk20106,reshetikhin2010lectures}. There are only three previous special cases, where a full description of the set of limit shapes has been obtained. For the free-fermion subvariety in the space of the parameters, the model is equivalent to random domino tilings (cf.\ \cite{ferrari2006domino}) and in this language the model was analyzed in \cite{CohnKenyonPropp2000, kenyon2007limit,BufetovKnizel,kenyon2024limit}, see also \cite{gorin2021lectures}.
In \cite{deGierKenyonWatson2021limit} and \cite{KenyonPrausegenus0} limit shapes for the case $\Delta\to\infty$ (the so-called five vertex model) were worked out.
For stochastic weights and free boundary conditions, the model is equivalent to an interacting particle system, and this fact
was used to compute limit shapes in \cite{gwa1992six,borodin2016stochastic,borodin2019stochastic,aggarwal2020limit}. From this perspective, our key discovery is that there exists a fourth case, where the six-vertex model is equivalent to another system (random permutations) and its limit shapes can be found.

An interesting accompanying question, which we do not address in this work, is how to design efficient perfect samplers for configurations of the six-vertex model with $c$ close to $0$ or for random Mallows restricted permutations. In the six-vertex literature a standard approach is to run a local Glauber dynamics until it mixes, as in \cite{keating2018random, fahrbach2019slow, belov2022two, lyberg2023fluctuation, prahofer2023domain}; in the random permutations literature, sampling restricted permutations by multiplying with random transpositions was discussed in \cite{diaconis2001statistical,bormashenko2011permutations}. However, for general $q$ and the domains we consider here the existing algorithms can be prohibitively slow.

The case of unrestricted random Mallows permutations is simpler. One can create the desired permutation by sampling $n-1$ truncated geometric random variables --- the procedure is called $q$--shuffle in \cite[Section 3]{gnedin2010q}. Alternatively, in \cite{diaconis2000analysis,benjamini2005mixing} it is noted that Mallows measure is a stationary distribution for colored Asymmetric Simple Exclusion Process (ASEP) with one particle of each color; hence, running colored ASEP for a long time delivers another straightforward sampler. While these approaches might be extended to certain simplest restrictions, the most general cases seem much harder. For instance \cite[Remark 3.5]{diaconis2001statistical} notices that the set of all permutations with fixed restrictions might fail to be connected by transpositions, thus, transposition-based Markov chains are not going to have a unique stationary distribution for such cases.

%and moreover evidence suggests that nonsimply connected regions do not typically have unique solutions.
\bigskip

\noindent{\bf Acknowledgments.}
We thank Istv\'an Prause and two anonymous referees for helpful comments.
Part of this research was performed while the authors were visiting the Institute for Pure and Applied Mathematics (IPAM), which is supported by the National Science Foundation (Grant No. DMS-1925919). V.G.\ was partially supported by NSF grant DMS - 2246449. R.K. was supported by NSF DMS-1940932 and the Simons Foundation grant 327929.

\section{Problem setup}

We present here two different points of view on the problem we study: limit shapes for the height function in an instance of the six-vertex model, permuton-type limits of restricted random permutations.

\subsection{Six-vertex model with rare corners} \label{Sections_setup_six}

A configuration in the six-vertex model in a finite domain $\Omega\subset \mathbb Z^2$ is an assignment to the vertices of this domain of one of the six types of Figure \ref{Figure_six_vertices} in a consistent way, i.e. so that the outcome is a collection of paths which are allowed to touch each other, but not to cross, and connecting points on the boundary of $\Omega$. Let $\partial \Omega$ be the collection of all edges of the grid, connecting $\Omega$ to its complement $\mathbb Z^2\setminus \Omega$. By \emph{boundary condition}, we mean a choice of $\{0,1\}$-sequence of length $|\partial\Omega|$, where each $1$ indicates that the corresponding edge of $\Omega$ is occupied by a path and $0$ indicates that the edge is not occupied. A configuration is consistent with boundary conditions, if the paths leave/enter domain $\Omega$ through the edges of $\partial \Omega$ labeled with $1$s and only through them.

\begin{figure}[t]
\begin{center}
   \includegraphics[width=0.8\linewidth]{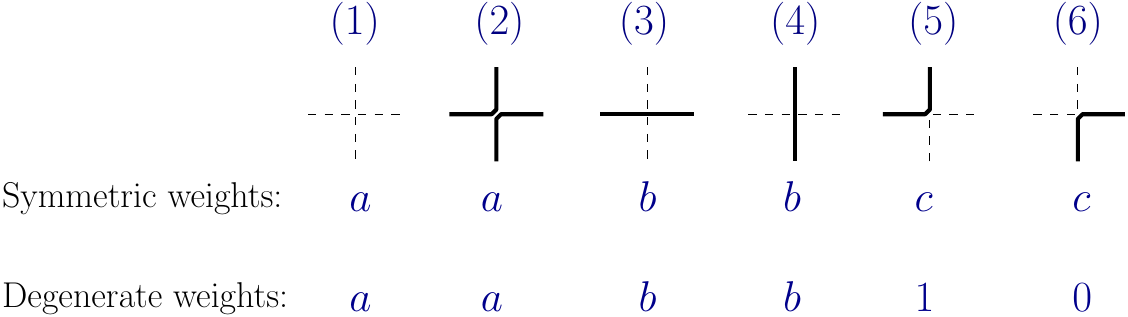}
\end{center}
        \caption{\label{Figure_six_vertices} The six types of vertices and their weights.}
\end{figure}

As a simple example, take $\Omega$ to be the $N\times N$ square, $\Omega=\{1,\dots,N\}\times\{1,\dots,N\}$ with $|\partial \Omega|=4N$ and the boundary condition such that all the edges along the left and top boundaries are occupied and all the edges along the bottom and right boundaries are not occupied, see Figure \ref{Figure_DWBC}. This situation is usually called the ``Domain Wall Boundary Condition''.

\begin{figure}[t]
\begin{center}
   \includegraphics[width=0.3\linewidth]{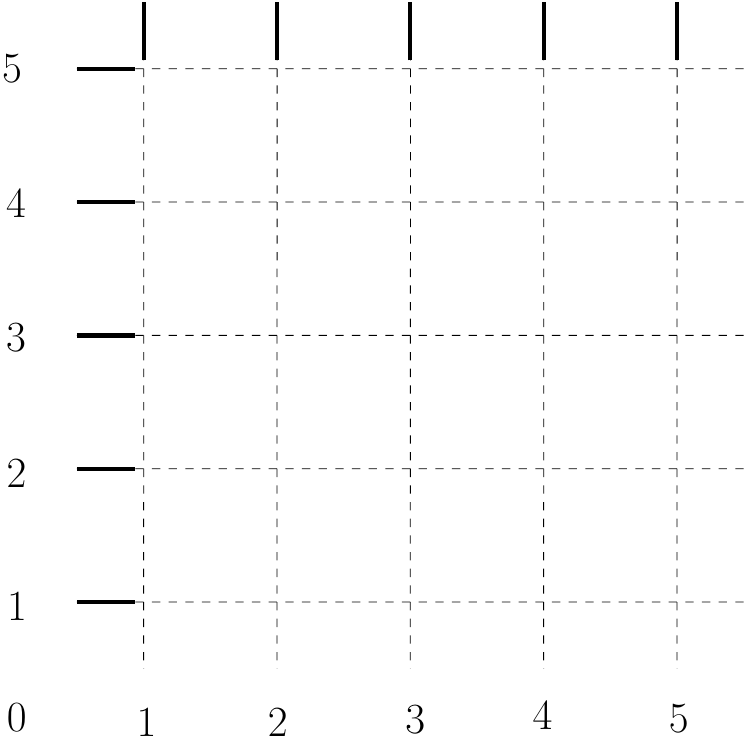} \quad \includegraphics[width=0.3\linewidth]{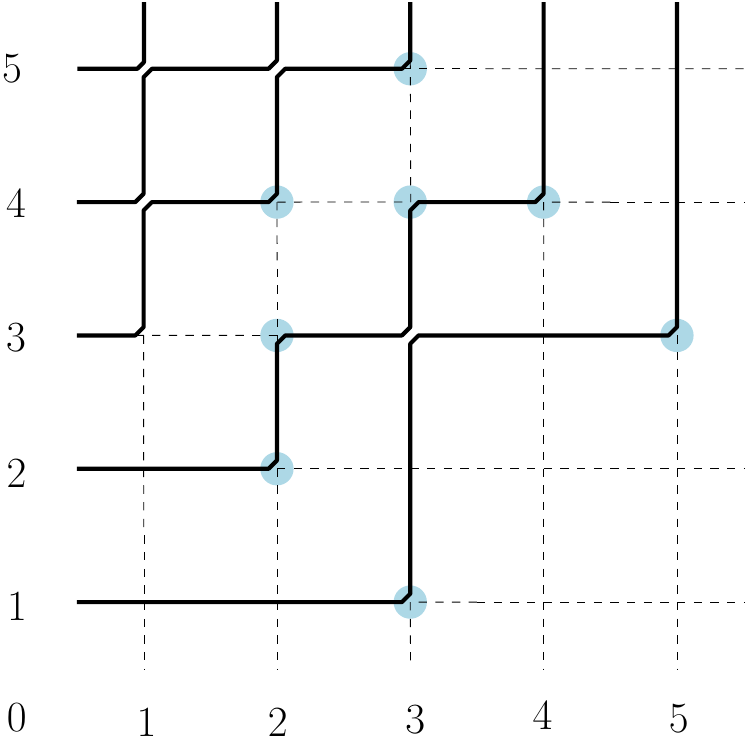} \quad
   \includegraphics[width=0.3\linewidth]{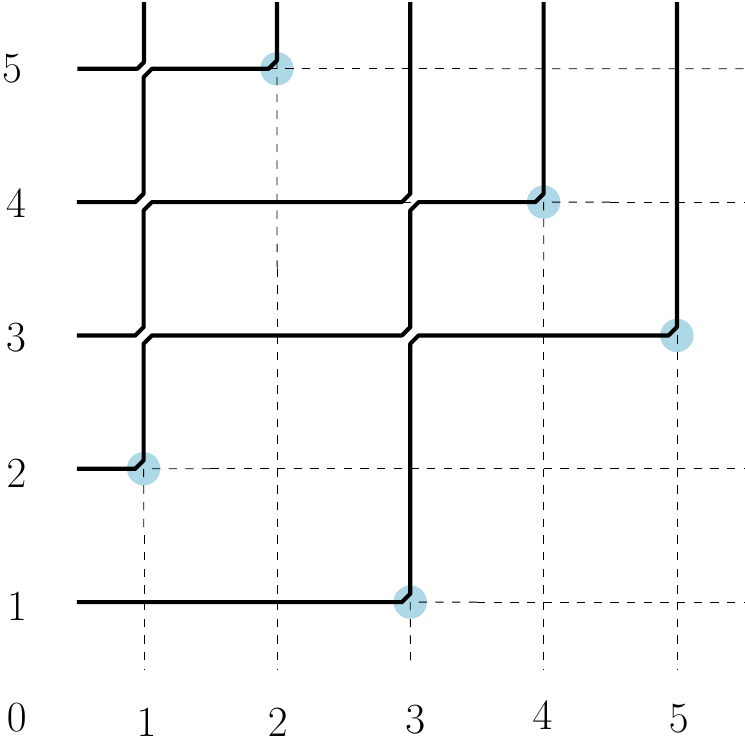}
\end{center}
        \caption{\label{Figure_DWBC} $N\times N$ square with Domain Wall Boundary Conditions for $N=5$ and two possible configurations of paths with $c$--type vertices emphasized.}
\end{figure}

\smallskip

For a more general class of domains, we need an auxiliary definition.

\begin{definition} \label{Definition_convex_array} A $k\times \ell$ array $\bigl[I_{uv}\bigr]_{1\le u\le k,\, 1\le v \le \ell}$ with $0/1$ entries is called \emph{convex}, if it is horizontally and vertically convex, that is, if for each horizontal\footnote{We use the Cartesian coordinate system for enumerating elements of the array $I$, rather than the usual matrix notation,
so the origin of the array is in the lower left corner.} coordinate $u$, the set $\{v\mid I_{uv}=1\}$ is a non-empty connected segment of integers, and for each vertical coordinate $v$, the set $\{u\mid I_{uv}=1\}$ is a non-empty connected segment of integers.
\end{definition}

\begin{remark}
 Any convex array can be constructed by starting from $k\times \ell$ array of all $1$s, choosing four Young diagrams adjacent to the four corners, and converting the numbers inside those Young diagrams into $0$s.
\end{remark}

\begin{example} The array $\begin{pmatrix} 1 & 1 &0 \\ 1& 1 & 1 \\ 0 & 1 & 0 \end{pmatrix}$ is convex, but $\begin{pmatrix} 1 & 1 &1 \\ 1& 1 & 0 \\ 0 & 1 & 1 \end{pmatrix}$ is not. \end{example}

We deal with domains parameterized by the following data:
\begin{itemize}
 \item Positive integers $k$, $\ell$, and $N$;
 \item Integers $0=X_0<X_1<\dots<X_k=N$ and $0=Y_0<Y_1<\dots<Y_\ell=N$;
 \item A convex array $\bigl[I_{uv}\bigr]_{1\le u\le k,\, 1\le v \le \ell}$.
\end{itemize}

\begin{definition} \label{Def_restrictions_6v}
We define $\Omega^{X,Y,I}$ to be a subset of $\{1,\dots,N\}^2$ given by
$$
  \Omega^{X,Y,I}=\{(x,y)\in\mathbb Z^2 \mid X_{u-1}<x\le X_u,\, Y_{v-1}<y\le Y_v\text{ for some } (u,v)\text{ such that }I_{uv}=1 \}.
$$
The boundary conditions are: horizontal paths enter into the left-most vertices in each of the $N$ rows; vertical paths leave the upper-most vertices in each of the $N$ columns.
\end{definition}

We refer to Figures \ref{Figure_domain_ex} and \ref{Figure_domain_ex_2} for examples of $\Omega^{X,Y,I}$, and remark that for $k=\ell=1$, this is exactly the $N\times N$ square with Domain Wall Boundary Conditions, as in Figure \ref{Figure_DWBC}.

\begin{figure}[t]
\begin{center}
   \includegraphics[width=0.35\linewidth]{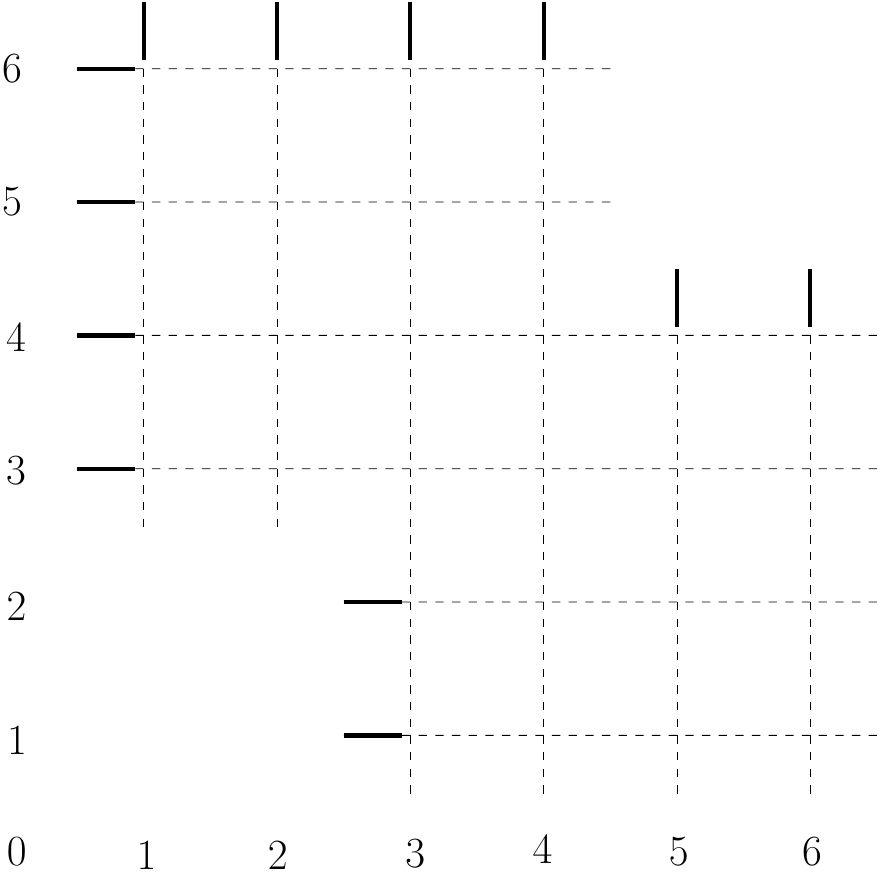} \hspace{1.5cm} \includegraphics[width=0.35\linewidth]{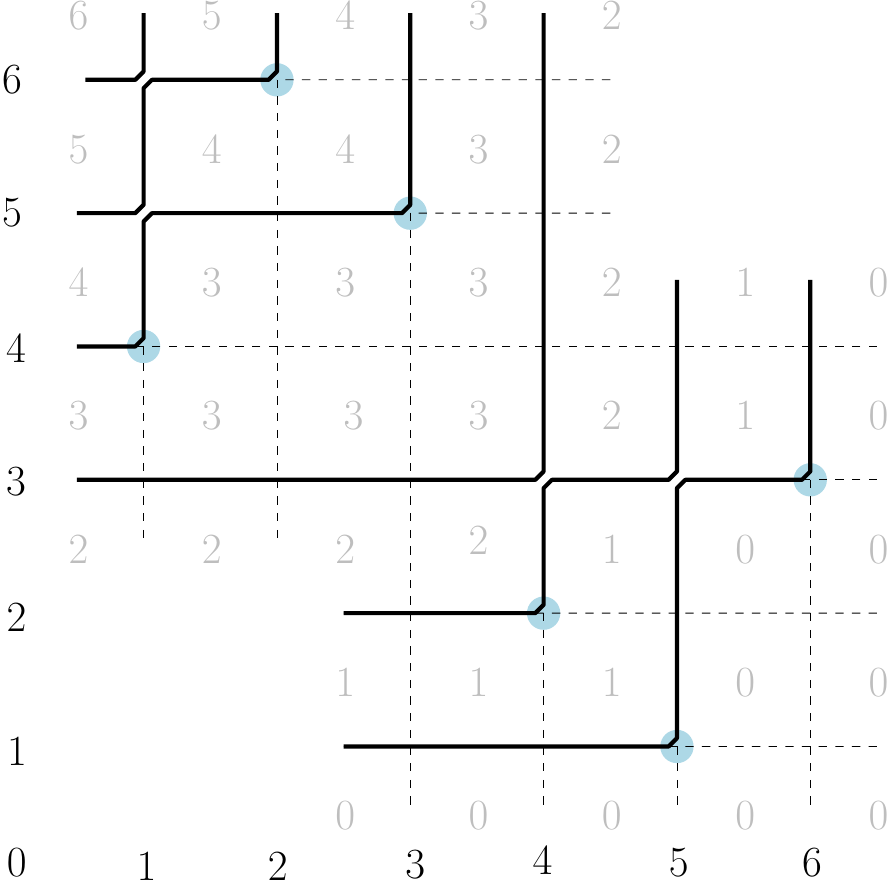}
\end{center}
\caption{\label{Figure_domain_ex} Domain $\Omega^{X,Y,I}$ with $N=6$, $k=\ell=3$, $X=(0,2,4,6)$, $Y=(0,2,4,6)$, $I=\protect\begin{pmatrix} 1 & 1 &0 \\ 1& 1 & 1 \\ 0 & 1 & 1 \protect\end{pmatrix}$, and one possible configuration of paths with $c$--type vertices emphasized and the values of the height function in light gray.}
\end{figure}

\begin{figure}[t]
\begin{center}
   \includegraphics[width=0.35\linewidth]{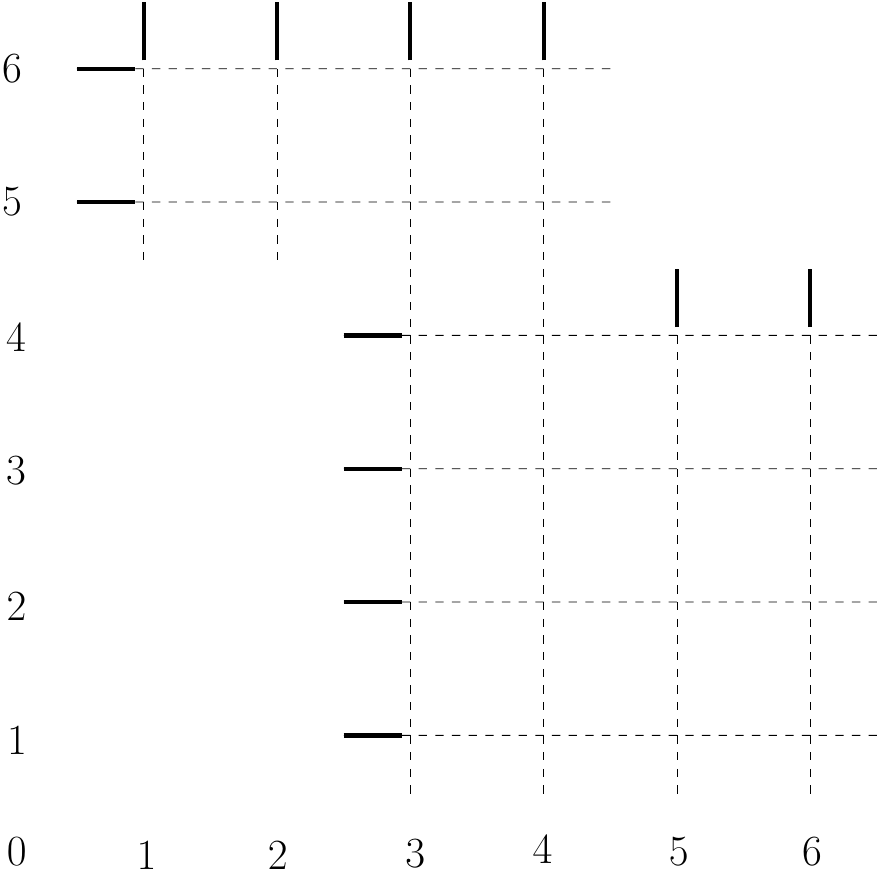} \hspace{1.5cm} \includegraphics[width=0.35\linewidth]{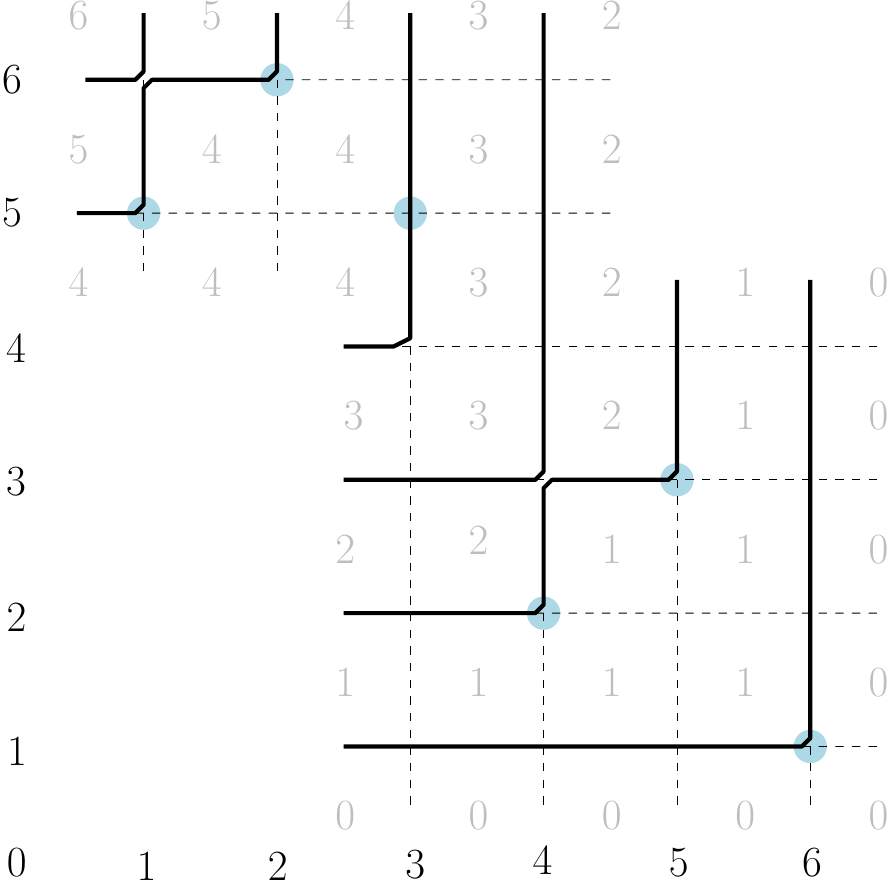}
\end{center}
\caption{\label{Figure_domain_ex_2} Domain $\Omega^{X,Y,I}$ as in Figure \ref{Figure_domain_ex}, but with $I=\protect\begin{pmatrix} 1 & 1 &0 \\ 0& 1 & 1 \\ 0 & 1 & 1 \protect\end{pmatrix}$.}
\end{figure}

\bigskip

Next, we introduce a weight $w_i>0$ for each vertex type $i = 1,2,\dots, 6$ and define the probability of a configuration $\sigma$ of the six-vertex model in domain $\Omega^{X,Y,I}$ as
\begin{equation}
\label{eq_6v_weight_general}
\mathrm{Prob}(\sigma) = \frac{1}{Z} \prod_{i=1}^6 w_i^{N_i(\sigma)},
\end{equation}
where $N_i(\sigma)$ is the number of vertices of Type $i$ in the configuration $\sigma$. The normalization $Z$ in \eqref{eq_6v_weight_general} is chosen so that the sum over all $\sigma$ in $\Omega^{X,Y,I}$ is $1$. One can show that for these deterministically fixed boundary conditions, the probability measure \eqref{eq_6v_weight_general} depends on two parameters, rather than six (there are four conservation laws, see e.g. \cite[Lemma 2.1]{Gorin_Nicoletti_lectures}). Hence, there is no loss of generality to consider six symmetric weights $(a,a,b,b,c,c)$, as in Figure \ref{Figure_six_vertices}.

We study random configurations of the six-vertex model in $\Omega^{X,Y,I}$ via their \emph{height functions}.

\begin{definition} \label{Definition_Height_function} The height function $H(x,y)$ is defined for $x,y\in \mathbb Z+\frac{1}{2}$ situated inside the domain as the total number of paths one crosses to get from $(x,y)$ to the bottom-right corner of the domain.
\end{definition}
We note that along the boundary of $\Omega$, the values of $H(x,y)$ do not depend on the choice of the configuration and can be seen as an alternative encoding of the boundary conditions, cf.\ Figures \ref{Figure_domain_ex} and \ref{Figure_domain_ex_2}.

\begin{question}
\label{Open_question}
What is the asymptotic behavior of $\frac{1}{N} H(Nx,Ny)$ for large $N$?
\end{question}
One hopes to have a deterministic limit $\mathfrak h(x,y)=\lim_{N\to\infty}\frac{1}{N} H(Nx,Ny)$, and we refer to it as the limit shape.  For generic values of $(a,b,c)$ the limit shape is unknown even for the $N\times N$ square of Figure \ref{Figure_DWBC}. The notable exception is on the so-called free fermionic subvariety: this means a choice of $(a,b,c)$ satisfying $a^2+b^2-c^2=0$. In this situation a bijection of the model with random domino tilings (see e.g., \cite{ferrari2006domino} and references therein) brings in various tools that can be used to find the limit shape $\mathfrak h(x,y)$, e.g.\ \cite{CohnKenyonPropp2000}, \cite{BufetovKnizel}, \cite{kenyon2024limit}.

We do not have anything new to say at generic $(a,b,c)$, but we study here the case $c=0$. In the limit $c\to 0$, the only surviving configurations of the model have the minimal possible number of vertices of Types $5$ and $6$. The specific form of $\Omega^{X,Y,I}$ and its boundary conditions guarantee that the minimum is achieved on the configurations which have $N$ Type $5$ vertices (one per row, one per column) and zero Type $6$ vertices. Hence, we can equivalently use the degenerate weights of Figure \ref{Figure_six_vertices}. In particular, the right configurations of Figures \ref{Figure_DWBC} and \ref{Figure_domain_ex} have positive probabilities in this limit, while the middle configuration of Figure \ref{Figure_DWBC} has vanishing probability and disappears.

\begin{remark}
 In our degeneration only five vertices remain. However, these are not the same five vertices as in \cite{deGierKenyonWatson2021limit}, and our model and limit shapes are very different from those.
\end{remark}

\bigskip

Our task is to describe the limit shape $\mathfrak h(x,y)$ for $\Omega^{X,Y,I}$ with degenerate weights.

\subsection{Restricted random permutations}  \label{Sections_setup_perm}

We restart and give a different definition of our objects of interest. We deal with random permutations $\sigma\in \mathfrak S_N$ of $N$ letters. There are many ways to introduce a probability distribution on permutations. We fix real $q>0$ and focus on the celebrated \emph{Mallows measure}:
\begin{equation}
\label{eq_Mallows}
 \mathrm{Prob}(\sigma)=\frac{1}{\mathcal Z} q^{\inv(\sigma)}, \qquad \inv(\sigma)=\#\bigl\{i<j\mid \sigma(i)>\sigma(j)\bigr\}.
\end{equation}
Such a measure was originally introduced in \cite{mallows1957non} and since that time it has been intensively studied in probability, statistics, and theoretical physics. For instance, \cite{he2022cycles} contains one recent result and many references to previous work.
Note that when $q=1$ this is the \emph{uniform measure} on permutations.

In \eqref{eq_Mallows} a further choice can be made: either $\sigma$ can be any of the $N!$ permutations, or there can be additional restrictions on $\sigma$. The restrictions we deal with are parameterized by the following data abbreviated $\Omega^{X,Y,I}$:
\begin{itemize}
 \item Positive integers $k$, $\ell$, and $N$;
 \item Integers $0=X_0<X_1<\dots<X_k=N$ and $0=Y_0<Y_1<\dots<Y_\ell=N$;
 \item A $0/1$ array $\bigl[I_{uv}\bigr]_{1\le u\le k,\, 1\le v \le \ell}$.
\end{itemize}

\begin{definition} \label{Def_restrictions_perm}
 We say that a permutation $\sigma\in S_N$ is \emph{restricted} by $\Omega^{X,Y,I}$, if for each $m=1,\dots,N$, we have $I_{u(m),v(m)}=1$, where
 \begin{itemize}
  \item $u(m)$ is the index $1\le u \le k$, such that $X_{u-1}<\sigma(m)\le X_u$; and
  \item $v(m)$ is the index $1\le v\le k$, such that $Y_{v-1}<m\le Y_v$.
 \end{itemize}
\end{definition}
\begin{remark} If one identifies $\sigma$ with a collection of points $\{(\sigma(m),m)\}_{m=1}^N$, then the condition says that all these points must be in $\Omega^{X,Y,I}$ of Definition \ref{Def_restrictions_6v}.
\end{remark}

In particular, if $I$ is a matrix of all ones, then there are no restrictions and all $N!$ permutations are possible; other choices of $I$ introduce various constraints, see Figure \ref{Figure_permutations_6} for an example. We fix $(X,Y,I)$ and analyze random permutations restricted by $\Omega^{X,Y,I}$ and sampled according to \eqref{eq_Mallows}. Such restricted random permutations play an important role in developing various statistical tests; we refer to \cite{diaconis2001statistical} for a review.
\begin{remark}
 If $I$ is an array of all ones, i.e., there are no restrictions, then $\mathcal Z$ in \eqref{eq_Mallows} is explicit: $\mathcal Z=(1+q)(1+q+q^2)\cdots (1+q+\dots q^{N-1})$. For generic $(X,Y,I)$ we do not expect any explicit formula for $\mathcal Z$.
\end{remark}

\begin{figure}[t]
\begin{center}
   \includegraphics[width=0.35\linewidth]{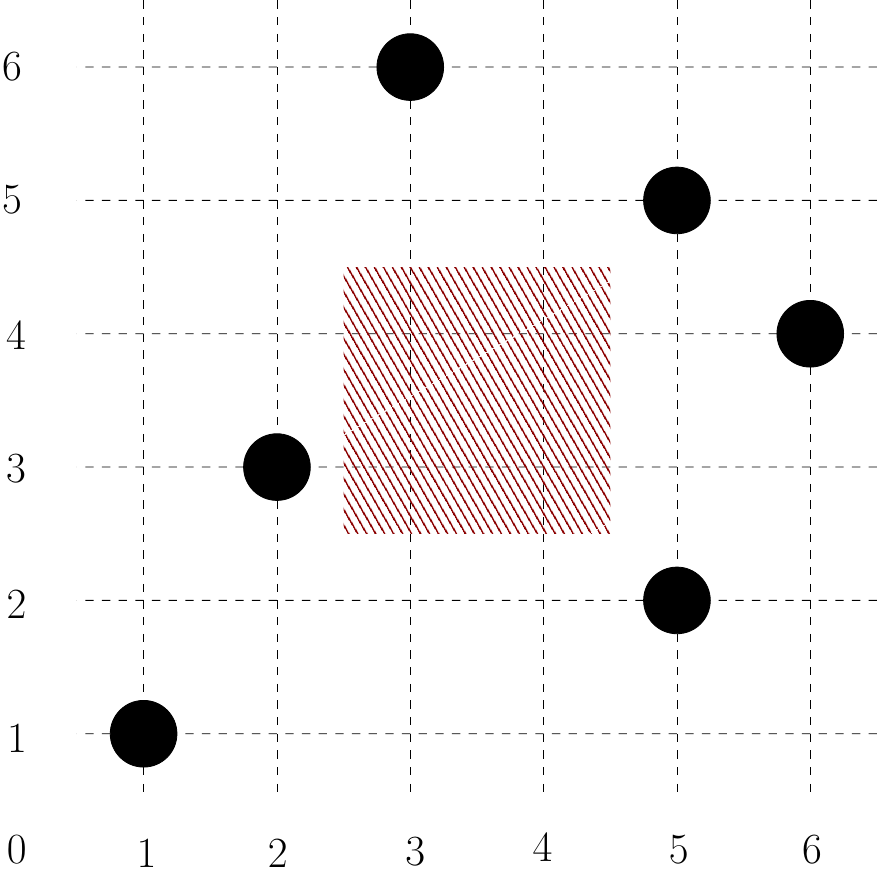} \hspace{2cm} \includegraphics[width=0.35\linewidth]{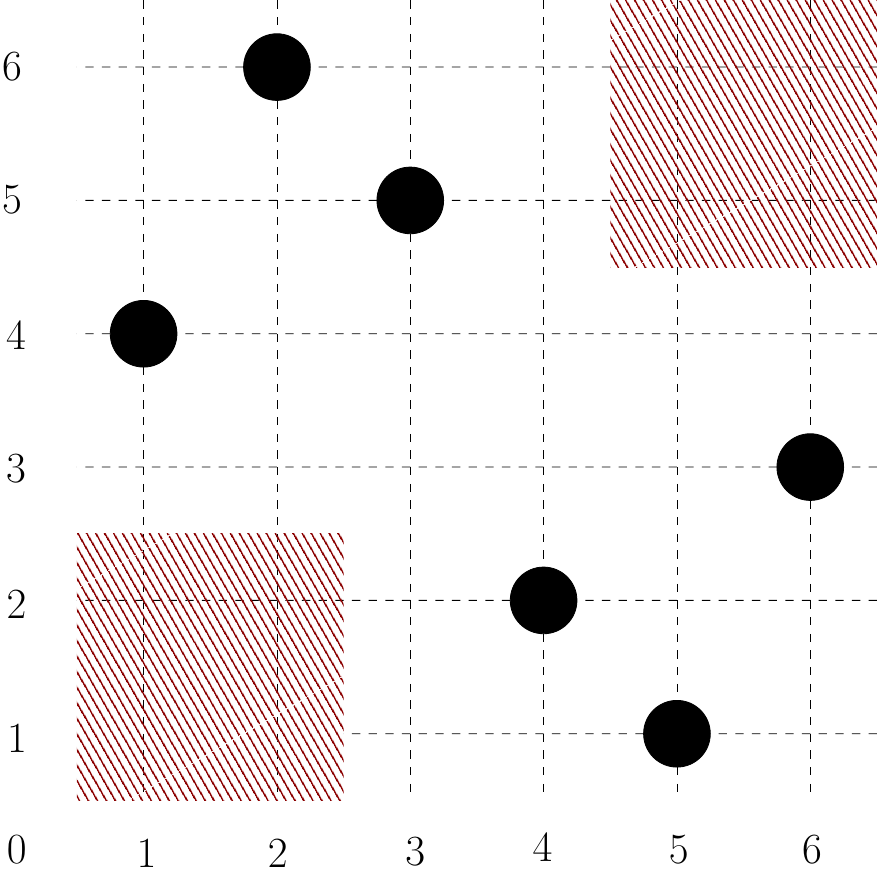}
\end{center}
\caption{\label{Figure_permutations_6} Possible permutations for two choices of $\Omega^{X,Y,I}$ with $N=6$, $k=\ell=3$, $X=(0,2,4,6)$, $Y=(0,2,4,6)$. Points  $(\sigma(m),m)$
are prohibited in red shaded regions. $I=\protect\begin{pmatrix} 1 & 1 &1 \\ 1& 0 & 1 \\ 1 & 1 & 1 \protect\end{pmatrix}$ on the left and $I=\protect\begin{pmatrix} 1 & 1 &0 \\ 1& 1 & 1 \\ 0 & 1 & 1 \protect\end{pmatrix}$ on the right with $\sigma=(152653)$ and $\sigma=(546132)$, respectively.
}
\end{figure}

We recall that a \emph{permuton} is a probability measure $\mu$ on $[0,1]^2$ with uniform marginals, i.e.\ such that $\mu([0,1]\times [0,y])=y$ and $\mu([0,x]\times [0,1])=x$ for all $0\le x,y\le 1$. Each permutation $\sigma\in \mathfrak S_N$ can be encoded by a permuton $\mu_\sigma$ of density
\begin{equation}
\label{eq_def_permuton}
 \mu_{\sigma}(x,y)= N \sum_{m=1}^N \mathbf 1_{\sigma(m)-1<x\le \sigma(m)} \mathbf 1_{m-1<y\le m}.
\end{equation}
In words, each point $(\sigma(m),m)$ gives rise to a $\tfrac{1}{N}\times \tfrac{1}{N}$ square of density $N$ for $\mu_\sigma$.

A random permutation $\sigma$ restricted by $\Omega^{X,Y,I}$ and sampled according to \eqref{eq_Mallows} gives rise to a random permuton $\mu_\sigma$. Our goal is to understand the limits of these permutons as $N\to\infty$. Due to \cite{starr2009thermodynamic}, the answer is known for unrestricted permutations (i.e.\ when $I$ is the array filled with $1$s), and we would like to understand how it changes when we introduce restrictions.

\subsection{Between the six-vertex and permutation points of view} Let us explain how the settings of Sections \ref{Sections_setup_six} and \ref{Sections_setup_perm} are interconnected.

\begin{theorem} \label{Theorem_correspondence}
 Take the $(X,Y,I)$ data of Definitions \ref{Def_restrictions_6v}, \ref{Def_restrictions_perm} with \emph{convex} array $I$ and assume that there exists at least one permutation restricted by $\Omega^{X,Y,I}$. Consider the correspondence between configurations $\sigma^{6v}$ of the six-vertex model in domain $\Omega^{X,Y,I}$ with degenerate weights of Figure \ref{Figure_six_vertices} (i.e.\ without Type $6$ vertices) and permutations $\sigma^{\text{perm}}$ restricted by $\Omega^{X,Y,I}$, identifying the Type $5$ vertices in $\sigma^{6v}$ with points $\{(\sigma^{\text{perm}}(m),m)\}_{m=1}^N$. This correspondence is a weight-preserving bijection between probability measures \eqref{eq_6v_weight_general} and \eqref{eq_Mallows} with $q=\frac{b^2}{a^2}$.
\end{theorem}
\begin{remark}
 For non-convex $I$, the restricted permutations make perfect sense, but the corresponding object on the six-vertex side is less natural. We would need to keep track of the six-vertex configurations inside the holes in $\Omega$ (cf.\ the left panel in Figure \ref{Figure_permutations_6}) and restrict to have no $c$-type vertices there.
\end{remark}
\begin{proof}[Proof of Theorem \ref{Theorem_correspondence}]
 Consider a configuration of the six-vertex model in domain $\Omega^{X,Y,I}$. Let us trace a row of the domain from left to right. Note that on the left border there is an incoming horizontal path, while on the right border there is none. Hence, there should be at least one vertex of Type $5$ in this row. By the same argument there should be at least one vertex of Type $5$ in each of $N$ columns. Therefore, the minimum number of vertices of Types $5$ and $6$ is $N$ and it is achieved by having exactly one Type $5$ vertex in each row and column (and no vertices of Type $6$). Hence, we arrive at the model with degenerate weights of Figure \ref{Figure_six_vertices} and with configuration of the Type $5$ vertices forming a permutation matrix. It is straightforward to check that these $N$ Type $5$ vertices uniquely determine the types of all other vertices and that every permutation restricted  by $\Omega^{X,Y,I}$ can be obtained in this way. This is our bijection between degenerate six-vertex configurations and permutations; it remains to check the correspondence between the weights  \eqref{eq_6v_weight_general} and \eqref{eq_Mallows}.

 For that we extend a configuration of the six-vertex model in domain $\Omega^{X,Y,I}$ to the configuration in $N\times N$ square with Domain Wall Boundary conditions. We note that $\{1,\dots,N\}^2\setminus \Omega^{X,Y,I}$ is a disjoint union of four connected domains adjacent to the four corners on the square (in Figure \ref{Figure_domain_ex}, the top--left and bottom-right domains are empty, while the top-right and bottom--left domains are $2\times 2$ squares), add Type $1$ vertices in all positions of the bottom--right domain, add Type $2$ vertices in the top--left domain, add Type $3$ vertices in the bottom--left domain, and Type $4$ vertices in the top--right domain. Since the additional vertices are exactly the same for each configuration in $\Omega^{X,Y,I}$, this extension does not change the probability measure \eqref{eq_6v_weight_general}.\footnote{This step would not have worked the same way for non-convex $I$.}

In order to compute the weight of the configuration in the full $N\times N$ square, we repeat the argument of \cite[Proposition 7.2]{Gorin_Liechty_2023}: In row $k$ of the configuration corresponding to the permutation $\sigma$ there is exactly one Type $5$ vertex and:
\begin{itemize}
 \item $N-\sigma(k)- \#\{i<k: \sigma(i)>\sigma(k)\}$ vertices of Type $1$,
 \item $\#\{i<k: \sigma(i)<\sigma(k)\}$ vertices of Type $2$,
 \item $\sigma(k)-1-\#\{i<k: \sigma(i)<\sigma(k)\}$ vertices of Type $3$,
 \item $\#\{i<k: \sigma(i)>\sigma(k)\}$ vertices of Type $4$.
\end{itemize}
Note that $\#\{i<k: \sigma(i)<\sigma(k)\}=k-1-\#\{i<k: \sigma(i)>\sigma(k)\}$. Hence, the product of the weights (we use
the degenerate weights of Figure \ref{Figure_six_vertices}) of the vertices in row $k$ is
\begin{equation*}
 a^{N+k-1-\sigma(k)- 2\cdot \#\{i<k:\, \sigma(i)>\sigma(k)\}} \,  b^{\sigma(k)-k+2\cdot \#\{i<k:\, \sigma(i)>\sigma(k)\}}.
\end{equation*}
Multiplying over all $k=1,\dots,N$, and omitting a prefactor which does not depend on $\sigma$, we conclude that the probability of a configuration corresponding to $\sigma$ is proportional to
$$
 \left(\frac{b^2}{a^2}\right)^{\#\{1\le i<k\le N:\, \sigma(i)>\sigma(k)\}}= \left(\frac{b^2}{a^2}\right)^{\inv(\sigma)}. \qedhere
$$
\end{proof}

\begin{proposition} \label{Proposition_Height_to_Permuton} Under the correspondence of Theorem \ref{Theorem_correspondence}, the height function of Definition \ref{Definition_Height_function} is related to the permuton of \eqref{eq_def_permuton} by:
\begin{equation}
\label{eq_height_to_permuton}
 \frac{1}{N} H(Nx+\frac{1}{2},Ny+\frac{1}{2})= \int_0^y  \int_x^1  \mu_\sigma(s, t)\,ds\,d t.
\end{equation}
\end{proposition}
\begin{remark}
 The identity \eqref{eq_height_to_permuton} is valid for $x,y\in\left\{\tfrac{0}{N},\tfrac{1}{N},\dots, \tfrac{N}{N} \right\}$, but also can be taken as a definition of (piecewise-linear) extension of $H$ to real arguments.
\end{remark}
\begin{proof}[Proof of Proposition \ref{Proposition_Height_to_Permuton}] $H(Nx+\frac{1}{2},Ny+\frac{1}{2})$ is
the total number of paths one crosses to get from $\left(Nx+\frac{1}{2},Ny+\frac{1}{2}\right)$ to the bottom-right corner of the domain, which is equal to the total number of $c$--type vertices
situated in the down-right direction from $\left(Nx+\frac{1}{2},Ny+\frac{1}{2}\right)$. The latter number is the same as $N \int_0^y  \int_x^1  \mu_\sigma(s, t)\,d s\,d t$.
\end{proof}

\section{Asymptotic theorems}

\subsection{Asymptotic regime} \label{Section_asymptotic_regime}
We fix $k$, $\ell$, and matrix $I$. In addition, we fix rational parameters $0=\mathrm x_0<\mathrm x_1<\dots<\mathrm x_{k-1}<\mathrm x_k=1$ and $0=\mathrm y_0<\mathrm y_1<\dots<\mathrm y_{\ell-1}<\mathrm y_\ell=1$ let $M$ be the common denominator of all these parameters. We define  $0=X_0<X_1<\dots<X_k=N$ and $0=Y_0<Y_1<\dots<Y_\ell=N$ through
\begin{equation}
\label{eq_XY_rescaling}
 X_u=N \mathrm x_u\quad 0\le u \le k,\qquad Y_v= N \mathrm y_v,\quad  0\le v \leq \ell.
\end{equation}
$N$ will be a large parameter of the form $N=Mn$, $n=1,2,3,\dots$. This guarantees that all the numbers in \eqref{eq_XY_rescaling} are integers. When we write $N\to\infty$, what we mean is $N=Mn$ and $n\to\infty$. For the parameter $q=\frac{b^2}{a^2}$ we further assume
\begin{equation}\label{qr}
 q=\exp\left(- \frac{r}{N} \right), \qquad r\in\mathbb R.
\end{equation}

\subsection{Variational principle}
The large $N$ limits for configurations in the interpretations of either Sections \ref{Sections_setup_six} or \ref{Sections_setup_perm}, can be described via a variational principle which we now present. Taking into account the correspondence of Proposition \ref{Proposition_Height_to_Permuton}, we encode the configurations by the normalized height function:
\begin{equation}
 h_\sigma(x,y)=\int_0^y  \int_x^1  \mu_\sigma(s, t) \dd s \dd t, \qquad 0\le x,y\le 1.
\end{equation}
The function $h_\sigma(x,y)$ belongs to the following class of functions:
\begin{definition} $\mathfrak F$ is the class of all real functions $h(x,y)$, $0\le x,y,\le 1$, which are $1$--Lipshitz in both $x$ and $y$ variables, satisfy for all $0\le x_1<x_2\le 1$, $0\le y_1<y_2\le 1$, the inequality
$$
h(x_1, y_1)+h(x_2,y_2)-h (x_1,y_2)-h(x_2,y_1)\le 0,
$$
 and satisfy the boundary conditions:
$$
 h(x,0)=0,\qquad h(x,1)=1-x,\qquad h(0,y)=y,\qquad h(1,y)=0.
$$
\end{definition}
\begin{remark}
 Taking $y_1=0$, we conclude that each $h\in \mathfrak F$ is weakly decreasing in $x$. Taking $x_2=1$, we conclude that each $h\in \mathfrak F$ is weakly increasing in $y$.
\end{remark}

We equip $\mathfrak F$ with the uniform topology, i.e.\ the topology generated by the metric $d(f,g)=\sup_{0\le x,y\le 1} |f(x,y)-g(x,y)|$.
It is straightforward to check that the correspondence $h(x,y)\leftrightarrow \mu([x,1]\times[0,y])$ is a bijection between functions in $\mathfrak F$ and permutons. Under this bijection, the uniform topology on $\mathfrak F$ becomes the weak topology on permutons treated as probability measures on $[0,1]^2$.

When the configurations are restricted by $\Omega^{X,Y,I}$, the normalized height functions belong to a closed (and thus compact) subset of $\mathfrak F$.

\begin{definition}  \label{Definition_class_of_functions} Given integers $k$, $\ell$, a $k\times \ell$ array $I$ filled with $0$s and $1s$, and real parameters $0=\mathrm x_0<\mathrm x_1<\dots<\mathrm x_{k-1}<\mathrm x_k=1$ and $0=\mathrm y_0<\mathrm y_1<\dots<\mathrm y_{\ell-1}<\mathrm y_\ell=1$, we define $\mathfrak F^{\mathrm x,\mathrm y,I}$ to be a closed subset of $\mathfrak F$ consisting of all functions $h\in\mathfrak F$, such that whenever $I_{uv}=0$,
\begin{equation}
\label{eq_height_function_flat}
 h(\mathrm x_u,\mathrm y_v)+h(\mathrm x_{u-1},\mathrm y_{v-1})-h (\mathrm x_u,\mathrm y_{v-1})-h(\mathrm x_{u-1},\mathrm y_v)=0.
\end{equation}
\end{definition}
\begin{remark}
 In terms of the corresponding permuton, \eqref{eq_height_function_flat} says that the mass of the square $[\mathrm x_{u-1},\mathrm x_u]\times [\mathrm y_{v-1},\mathrm y_v]$ is zero.
\end{remark}

\begin{definition} \label{Definition_non_degenerate}
 We say that the data $\mathrm x_0<\mathrm x_1<\dots<\mathrm x_k$, $\mathrm y_0<\mathrm y_1<\dots<\mathrm y_\ell$, and $0/1$ array $I=[I_{uv}]$ of size $k\times \ell$ is \emph{non-degenerate} if there exists a non-negative real array $[B_{uv}]$, such that $B_{uv}>0$ if and only if $I_{uv}=1$ and $\sum_{v=1}^{\ell} B_{uv}=\mathrm x_u-\mathrm x_{u-1}$, for all $1\le u \le k$, and $\sum_{u=1}^k B_{uv}=\mathrm y_v-\mathrm y_{v-1}$ for all $1\le v \le \ell$.
\end{definition}
Note that $(\mathrm x,\mathrm y, I)$ being non-degenerate implies that the set $\mathfrak F^{\mathrm x,\mathrm y,I}$ is non-empty. In addition, this condition rules out the cases when the choice of $\mathrm x$ and $\mathrm y$ implies vanishing of the permuton on some rectangles; in this situation we can always change some elements of $I$ to zeros and achieve a non-degenerate situation.\footnote{For instance, take $k=\ell=2$, $\mathrm x=\mathrm y=(0,\tfrac{1}{2},1)$ and $I=\begin{pmatrix} 1 & 0 \\ 1 & 1 \end{pmatrix}$. This is a degenerate situation, because uniform marginals conditions implies that the permuton vanishes in the bottom-left square. Hence, we should replace $I$ with $\begin{pmatrix} 1 & 0 \\ 0 & 1 \end{pmatrix}$ to get a non-degenerate triplet corresponding to exactly the same set $\mathfrak F^{\mathrm x,\mathrm y,I}$.}  From now on we
tacitly assume that $(\mathrm x,\mathrm y, I)$ is non-degenerate.

We use the \emph{permuton energy} of the height functions to compute limit shapes:

\begin{definition} Given a function $h\in \mathfrak F^{\mathrm x,\mathrm y,I}$ we define its \emph{permuton energy} to be
\begin{equation}
\label{eq_energy}
 \mathfrak E (h)= \int_0^1\int_0^1 \bigl[ h_{xy} \ln(-h_{xy})-r h_x h_y\bigr]\, \dd x\, \dd y,
\end{equation}
(where $r$ is from (\ref{qr}))
if the integral exists and $ \mathfrak E (h)=-\infty$ otherwise. The product $h_{xy} \ln(-h_{xy})$ is understood as $0$ when $h_{xy}=0$.
\end{definition}
Using the correspondence, $h(x,y)\leftrightarrow \mu([x,1]\times[0,y])$, the first term in \eqref{eq_energy} is the entropy of the measure $\mu$. A proper way to think about this term is that it is given by the double Lebesgue integral $-\int_0^1\int_0^1 \mu(x,y) \ln(\mu(x,y))\, \dd x \, \dd y$, whenever $\mu$ has a density $\mu(x,y)$ with respect to the Lebesgue measure on $[0,1]^2$; the latter integral is either finite or $-\infty$. If $\mu$ is not absolutely continuous, then we set the integral to be $-\infty$ by definition. As for the second term in \eqref{eq_energy}, since $-1\le h_x \le 0$ and $0\le h_y \le 1$ by Lipshitz condition, the integral of this term always exists and is upper bounded by $|r|$ in magnitude.

\begin{theorem} \label{Theorem_LDP} In the asymptotic regime of Section \ref{Section_asymptotic_regime}, for any closed non-empty $A\subset \mathfrak F^{\mathrm x,\mathrm y,I}$ we have
\begin{equation}\label{eq_Probability_upper_bound}
 \limsup_{N\to\infty} \frac{1}{N}\ln\left(\frac{1}{N!}\sum_{\sigma\in \mathfrak S_N:~h_\sigma\in A} q^{\inv(\sigma)}\right)\le \sup_{h\in A} \mathfrak E(h)
\end{equation}
and for any non-empty open\footnote{We mean here open as a subset of $\mathfrak F^{\mathrm x,\mathrm y,I}$, rather than open as a subset of $\mathfrak F$.}  $A\subset \mathfrak F^{\mathrm x,\mathrm y,I}$, we have
\begin{equation} \label{eq_Probability_lower_bound}
 \liminf_{N\to\infty} \frac{1}{N}\ln\left(\frac{1}{N!}\sum_{\sigma\in \mathfrak S_N:~h_\sigma\in A}\right)\ge \sup_{h\in A}  \mathfrak E(h).
\end{equation}
\end{theorem}
\begin{remark}
 Since $A\subset \mathfrak F^{\mathrm x,\mathrm y,I}$, the condition $h_\sigma\in A$ guarantees that $\sigma$ is restricted by $\Omega^{X,Y,I}$. Hence, \eqref{eq_Probability_upper_bound}, \eqref{eq_Probability_lower_bound} count asymptotics for the probability measures \eqref{eq_6v_weight_general} with degenerate weights or \eqref{eq_Mallows}.
\end{remark}
We omit the proof for Theorem \ref{Theorem_LDP} and only mention that the first term in \eqref{eq_energy} is the entropy appearing from counting the permutations, while the second term is the logarithm of the weight $q^{\inv(\sigma)}$ reexpressed in terms of the height function $h_\sigma$. Statements very similar to Theorem \ref{Theorem_LDP} can be found in \cite{starr2009thermodynamic,trashorras2008large,mukherjee2016estimation,kenyon2020permutations,starr2018phase,borga2024large}. Note that in contrast to the variational principle for domino tilings of \cite{CohnKenyonPropp2000} (domino tilings are equivalent to configurations of the six-vertex model with weights satisfying $a^2+b^2=c^2$, see \cite{ferrari2006domino}), the prefactor in front of $\ln$ in \eqref{eq_Probability_upper_bound},\eqref{eq_Probability_lower_bound} is $\tfrac{1}{N}$ rather than $\tfrac{1}{N^2}$.

\begin{corollary}[Variational principle] \label{Corollary_limit_shape} In the asymptotic regime of Section \ref{Section_asymptotic_regime}, suppose that the permuton energy $\mathfrak E (h)$, $h\in\mathfrak F^{\mathrm x,\mathrm y,I}$, has a unique maximizer denoted $\mathfrak h$. Then, with $\sigma$ distributed according to \eqref{eq_6v_weight_general} with degenerate weights or \eqref{eq_Mallows}, for each $\eps>0$ we have
$$
 \lim_{N\to\infty} \mathrm{Prob}\left(\sup_{0\le x,y\le 1} |h_\sigma(x,y)-\mathfrak h(x,y)|<\eps\right)=1.
$$
\end{corollary}
\begin{proof} We apply Theorem \ref{Theorem_LDP} to the $\eps$--neighborhood of $\mathfrak h$ and its complement.
\end{proof}
\begin{remark}
 If we use the interpretation of Section \ref{Sections_setup_perm} in terms of random permutations, then the convexity of $I$ is not used in Corollary \ref{Corollary_limit_shape}.
\end{remark}

In the rest of the paper we concentrate on identifying the maximizers of $\mathfrak E(h)$.

\subsection{Maximizers via Euler-Lagrange equations}
We start the analysis of the variational problem $\mathfrak E(h)\to \max$, $h\in \mathfrak F^{\mathrm x,\mathrm y,I}$.
We don't assume that $I$ is convex unless explicitly stated.

We say that a point $(x,y)$, $0< x,y< 1$, is strictly inside the (rescaled) domain $\Omega^{\mathrm x,\mathrm y,I}$, if $\mathrm x_{u-1}< x< \mathrm x_u,$ and $\mathrm y_{v-1}<y< \mathrm y_v$ for some  $(u,v)$  such that $I_{uv}=1$. Note that we prohibit the end-points $x=x_u$, $y=y_v$, etc.

\begin{proposition}[Four-point relation] \label{Proposition_EL} Let $h\in \mathfrak F^{\mathrm x,\mathrm y,I}$ be a maximizer of $\mathfrak E$ and set $g=-h_{xy}$ to be the density of the corresponding permuton. Take $0<x_1<x_2<1$, $0<y_1<y_2<1$, such that the four points $(x_1,y_1)$, $(x_1,y_2)$, $(x_2,y_2)$, $(x_2,y_1)$ are strictly inside $\Omega^{\mathrm x,\mathrm y,I}$ and $g(x,y)$ is continuous and positive at these points. Then
\begin{equation}
\label{eq_rectangle_condition}
 \ln\left( \frac{g(x_1,y_1)g(x_2,y_2)}{g(x_1,y_2) g(x_2,y_1)}\right)-2r \int_{y_1}^{y_2}\int_{x_1}^{x_2} g(x,y)\, \dd x\, \dd y=0.
\end{equation}
\end{proposition}
\begin{proof}
 We choose two small constants $\eps\in\mathbb R$ and $\delta>0$, consider four squares of side-length $2\delta$ centered at $(x_1,y_1)$, $(x_1,y_2)$, $(x_2,y_2)$, $(x_2,y_1)$, add $\eps$ to the values of $g$ in the first and third one and subtract $\eps$ from the values in the second and fourth ones. Note that the boundary conditions on $h$ are unchanged under this transformation and let us compute the change in the double integral expression \eqref{eq_energy} for the energy $\mathfrak E (h)$.

 Differentiating the first term in the double integral in $g=-h_{xy}$, we conclude that it changes by
 \begin{equation}\label{eq_x5}
  -4 \delta^2 \eps \bigl[\ln g(x_1,y_1) + \ln g(x_2,y_2) -  \ln g(x_1,y_2)-  \ln g(x_2,y_1)+o(1)\bigr],
 \end{equation}
 as $\eps,\delta\to 0$. For the second term, we notice that
 $$
 h_x=-\int_0^y g(x,t) \, \dd t, \qquad h_y=\int_x^1 g(s,y) \, \dd s.
 $$
 Hence, the integrand in the second term changes only in $\delta$--neighborhood of the  $(x_1,y_1)-(x_1,y_2)-(x_2,y_2)-(x_2,y_1)$ rectangle.
 Inside the corner four squares of side-length $2\delta$, $h_x$ and $h_y$ change by $O(\eps\delta)$. Therefore, the total change in the part of the double integral where we integrated over these squares is $O(\eps\delta^3)$. This term will be negligible eventually. More important is the change of $h_x$ and $h_y$ in the $\delta$--neighborhood of the sides. This change is:
 $$
  2\delta \int_{x_1}^{x_2} (2r \eps\delta h_x(x,y_1)-2r\eps\delta h_x(x,y_2))\, \dd x\\+2\delta \int_{y_1}^{y_2}( 2r\eps \delta h_y(x_1,y)-2r\eps\delta h_y(x_2,y))\, \dd y + o(\eps\delta^2),
 $$
 which simplifies to
 \begin{multline}
 \label{eq_x6}
  4r\delta^2 \eps\left( \int_{x_1}^{x_2}\int_{y_1}^{y_2} g(x,y)\, \dd y\, \dd x + \int_{y_1}^{y_2} \int_{x_1}^{x_2} g(x,y)\,\dd x\, \dd y + o(1)\right)
\\=  8r\delta^2 \eps\left( \int_{x_1}^{x_2}\int_{y_1}^{y_2} g(x,y)\, \dd y\, \dd x + o(1)\right).
 \end{multline}
 The total change is the sum of \eqref{eq_x5} and \eqref{eq_x6}. Since $\eps$ can be of both signs, the sum should be zero in the leading order, as otherwise $h$ is not a maximizer. This gives the equation
 \begin{multline*}
   -4 \delta^2 \eps \bigl[\ln g(x_1,y_1) + \ln g(x_2,y_2) -  \ln g(x_1,y_2)-  \ln g(x_2,y_1)\bigr]\\+8r\delta^2 \eps\left( \int_{x_1}^{x_2}\int_{y_1}^{y_2} g(x,y)\, \dd y\, \dd x \right)=0.\qedhere
 \end{multline*}
\end{proof}

%%%
Proposition \ref{Proposition_EL} applies only at the points where $g$ is positive. Yet, in all the examples we checked, $g$ is actually positive everywhere inside the domain. We state this as a conjecture.
\begin{conjecture}\label{Conjecture_positivity}
 Suppose that the data is non-degenerate in the sense of Definition \ref{eq_rectangle_condition}. Let $h\in \mathfrak F^{\mathrm x,\mathrm y,I}$ be a maximizer of $\mathfrak E$ and set $g=-h_{xy}$ to be the density of the corresponding permuton. Then there exist constants $0<c<C$ (which depend on all the data) such that $c<g(x,y)<C$ for almost all points $(x,y)\in \Omega^{\mathrm x,\mathrm y,I}$.
\end{conjecture}
\begin{remark}
 In contrast to our $c\to 0$ situation, for the six-vertex model with generic parameters $a,b,c$ positivity of the density of the $c$--type vertices  does not hold everywhere. Instead, there are interesting arctic curves bounding the regions with no $c$--type vertices, see e.g.\ \cite{jockusch1998random,colomo2010arctic,aggarwal2020arctic}.
\end{remark}

%%%
\old{
\subsection{Positivity: proof of Proposition \ref{Proposition_positivity}}
\label{Section_positivity}

I can prove that on each rectangle $(\mathrm x_{u-1},\mathrm x_u)\times (\mathrm y_{v-1},\mathrm y_v)$, either $g(x,y)$ is identical zero, or $g(x,y)$ is bounded away from $0$. This argument only uses the 4-point relation.

Then the problem becomes a finite-dimensional combinatorial statement. Starting from this, I checked for $k=\ell=3$ and various choices of $I$, that $g$ can not vanish, but the arguments are delicate - they use all of uniform marginals, 4-point relation, non-degeneracy of data...

%Start with a small rectangle and bound the integral of $g$ over it using integrated $4$--point relation.
}
%%%

The four-point relation (\ref{eq_rectangle_condition}) is the Euler-Lagrange equation
for the variational problem $\mathfrak E (h)\to \max$, $h\in \mathfrak F^{\mathrm x,\mathrm y,I}$.
In general, this variational problem is not convex, so additional care is required to identify the solutions of the Euler-Lagrange equation with the maximizers. However, for small $r$ we can still use convexity ideas. For that we define the concept of a smooth solution.

\begin{definition} \label{Definition_smooth_solution}
 We call a function $h\in \mathfrak F^{\mathrm x,\mathrm y,I}$ a \emph{smooth solution to the EL-equations of class $[c,C]$}, if there exists another (measurable) function $g:[0,1]^2\to\mathbb R_{\ge 0}$, such that:
 \begin{enumerate}
  \item $\displaystyle h(x,y)=\int_0^y  \int_x^1  g(s, t)\, d s\, d t, \qquad 0\le x,y\le 1.$
  \item For each $1\le u\le k$, $1\le v\le l$, the restriction of $g$ on the open rectangle $(\mathrm x_{u-1},\mathrm x_u)\times (\mathrm y_{v-1},\mathrm y_v)$ has a continuous extension $g^{uv}$ to the closed rectangle $[\mathrm x_{u-1},\mathrm x_u]\times [\mathrm y_{v-1},\mathrm y_v]$.
  \item If $I_{u,v}=0$, then $g^{uv}(x,y)=0$. \\ If $I_{u,v}=1$, then $c\le g^{uv}(x,y)\le C$ for all $\mathrm x_{u-1}\le x \le \mathrm x_u$, $\mathrm y_{v-1}\le y \le \mathrm y_v$.
  \item For each  $0<x_1\le x_2<1$, $0<y_1\le y_2<1$, we have
    \begin{equation}
\label{eq_rectangle_condition_2}
 \ln\left( \frac{g(x_1-,y_1-)g(x_2+,y_2+)}{g(x_1-,y_2+) g(x_2+,y_1-)}\right)-2r \int_{y_1}^{y_2}\int_{x_1}^{x_2} g(x,y)\,dx\,dy=0,
\end{equation}
 where the $\pm$ signs are the left and right limits, e.g.\ $g(x_1-,y_1-)=\lim\limits_{\eps\to 0+} g(x_1-\eps,y_1-\eps)$.
 \end{enumerate}
\end{definition}

Note that we include the cases of degenerate rectangles $x_1=x_2$ and/or $y_1=y_2$ in (\ref{eq_rectangle_condition_2}) above.
\begin{theorem} \label{Theorem_maximizer_small_r} Fix two real constants $0<c<C$, array $I$, and real parameters $0=\mathrm x_0<\mathrm x_1<\dots<\mathrm x_{k-1}<\mathrm x_k=1$ and $0=\mathrm y_0<\mathrm y_1<\dots<\mathrm y_{\ell-1}<\mathrm y_\ell=1$. There exists a real constant $r_0>0$ (depending on all the data from the previous sentence), such that for all $r$ satisfying $|r|<r_0$, if $\hat h \in \mathfrak F^{\mathrm x,\mathrm y,I}$ is a smooth solution to the
EL equations of class $[c,C]$, then $\hat h$ is the unique solution of the variational problem $\mathfrak E(h)\to \max$, $h\in \mathfrak F^{\mathrm x,\mathrm y,I}$.
\end{theorem}

Our computed examples of Sections \ref{Section_smooth} and \ref{Section_examples} are all smooth, and hence unique solutions for small $|r|$ by this result.

\begin{remark}
It is plausible that $r_0=+\infty$ in many situations. For instance, for the unrestricted Mallows measure (corresponding to $k=\ell=1$, $I=(1)$ case) this is known to be true, see \cite{starr2009thermodynamic, starr2018phase}.
\end{remark}

\begin{proof}[Proof of Theorem \ref{Theorem_maximizer_small_r}]
We fix $r$ and a parameter $\eps>0$. We consider the function $h\in \mathfrak F^{\mathrm x,\mathrm y,I}$ given by
$$
  h(x,y)=\hat h(x,y)+ \eps f(x,y).
$$
Our task is to show that whenever $f$ is not identical zero, we have
\begin{equation}
\label{eq_x7}
\mathfrak E(h)<  \mathfrak E(\hat h).
\end{equation}
We assume without loss of generality that $-h_{xy}$ (the density of the permuton corresponding to $h$) is well-defined as a density of an absolutely continuous probability measure; otherwise $\mathfrak E(h)=-\infty$ and \eqref{eq_x7} is clear.

 We expand in $\eps$ the integrand in the permuton energy \eqref{eq_energy}:
\begin{multline}
  \bigl(\hat h_{xy}+\eps f_{xy} \bigr) \ln\bigl(-\hat h_{xy}-\eps f_{xy} \bigr)-r \bigl(\hat h_x+\eps f_x \bigr) \bigl(\hat h_y+\eps f_y\bigr)\\
  = \hat h_{xy} \ln\bigl(-\hat h_{xy}\bigr)-r \hat h_x\hat h_y
   +\eps\left[f_{xy}  \ln\bigl(-\hat h_{xy}\bigr)+ \hat h_{xy} \cdot \frac{f_{xy}}{\hat h_{xy}}-r \hat h_x f_y - r f_x \hat h_y \right]
   \\+\eps^2 \left[ \hat h_{xy}\frac{\ln\left(1+\eps \frac{f_{xy}}{\hat h_{xy}}\right)-\eps\frac{f_{xy}}{\hat h_{xy}}}{\eps^2}+f_{xy}\frac{\ln\left(1+\eps \frac{f_{xy}}{\hat h_{xy}}\right)}{\eps}  -r f_x f_y\right].
\end{multline}
Note that the $\int_0^1 \int_0^1$ integral of the term $\eps\bigl[\cdot\bigr]$ vanishes, because $\hat h$ is a solution to the Euler--Lagrange equations: we provide more detail in Lemma \ref{Lemma_EL_vanish} below.
 Hence, we are interested in $\eps^2\bigl[\cdot\bigr]$ term and we would like to show that its $\int_0^1 \int_0^1$ integral is always negative. The particular value of $\eps$ is not important (it can be absorbed into the definition of $f$) and we will set $\eps=1$ from now on.
 Integrating by parts in the $-r f_x f_y$ term and noticing that boundary terms vanish because $h$ and $\hat h$ satisfy the same boundary conditions, we want to prove that

 \begin{equation} \label{eq_x16}
  \int_0^1 \int_0^1 \S^{[2]} \dd x \dd y \stackrel{?}{<} 0,  \qquad \S^{[2]}:=  \hat h_{xy}\left(1+\frac{f_{xy}}{\hat h_{xy}}\right)\ln\left(1+ \frac{f_{xy}}{\hat h_{xy}}\right)-f_{xy} +r f f_{xy}.
 \end{equation}
 Let us denote
 \begin{equation}
 \label{eq_mu_def}
  \nu=\frac{f_{xy}}{\hat h_{xy}},
 \end{equation}
 so that
 \begin{equation}\label{eq_x17}
  f(x,y)=-\int_0^y\int_x^1 \nu(s,t) \hat h_{xy}(s,t)\, d s \, d t, \qquad 0\le x,y\le 1.
 \end{equation}
 and
 $$
 \S^{[2]}=  \hat h_{xy}\bigl[\left(1+\nu\right)\ln\left(1+ \nu\right)-\nu +r f \nu\bigr].
 $$

 We will use the following statement whose proof we postpone:
 \begin{lemma}\label{Lemma_product_bound} We have
  $$\left|\int_0^1 \int_0^1 f(x,y) \nu(x,y)\, \dd x \, \dd y \right|\le C  \iint \nu^2 \mathbf 1_{|\nu|\le 1/2} \, dx\, dy +
 \frac{4C}{c}\iint   |\nu| \mathbf 1_{|\nu|>1/2} \, dx\, dy.$$
 \end{lemma}

\smallskip

 Simultaneously, we use the following elementary bound, valid for all $z>-1$ (and in which we have not tried to make the constants optimal):
 $$
  (1+z)\ln(1+z)-z\ge \frac{z^2}{4}\mathbf 1_{|z|\le 1/2}+ \frac{|z|}{10} \mathbf 1_{|z|>1/2}.
 $$
 Therefore, noting that $\hat h_{xy}$ is negative and $\le -c$, we get
 \begin{multline}
 \label{eq_x19}
  \int_0^1\int_0^1 S^{[2]}\, dx\, dy\le -c \int_0^1\int_0^1 \left[ \frac{\nu(x,y)^2}{4}\mathbf 1_{|\nu|\le 1/2}+ \frac{|\nu(x,y)|}{10} \mathbf 1_{|\nu|>1/2}\right]\, dx\, dy \\ + |r|  C  \iint \nu(x,y)^2 \mathbf 1_{|\nu|\le 1/2} \, dx\, dy +
 |r|\frac{4C}{c}\iint   |\nu(x,y)| \mathbf 1_{|\nu|>1/2}\, dx\, dy.
 \end{multline}
 Given the values of $C$ and $c$, we can choose $r_1=r_1(C,c)$, so that the last expression is negative whenever $|r|<r_1$ and $\nu$ is different from identical zero.
\end{proof}
\begin{proof}[Proof of Lemma \ref{Lemma_product_bound}]
 Using \eqref{eq_x17}, we have
 \begin{equation}
 \label{eq_x18}
 \left|\int_0^1 \int_0^1 f(x,y) \nu(x,y)\, dx\, dy \right|\le C \int_0^1 \int_0^1 \int_0^y\int_x^1 |\nu(s,t)| |\nu(x,y)|\,  ds \, dt\, dx \, dy
 \end{equation}
 We split the integral into three parts and then sum the estimates:
 \begin{itemize}
  \item If both $|\nu(s,t)|\le 1/2$ and $|\nu(x,y)|\le 1/2$, then we write
  $$
   |\nu(s,t)| |\nu(x,y)|\le \frac{1}{2}\left(\nu^2(s,t)+\nu^2(x,y)\right),
  $$
  plug into \eqref{eq_x18} and integrate over $s,t,x,y$ to get an upper bound
  \begin{multline*}
    C \int_0^1 \int_0^1 \int_0^y\int_x^1 \mathbf 1_{|\nu(s,t)|\le 1/2} \mathbf 1_{|\nu(x,y)|\le 1/2} \frac{1}{2}\left(\nu^2(s,t)+\nu^2(x,y)\right) \,  ds \, dt\, dx \, dy
     \\ \le C  \int_0^1\int_0^1 \nu^2(x,y) \mathbf 1_{|\nu(x,y)|\le 1/2} \, dx\, dy.
  \end{multline*}
  \item If $|\nu(x,y)|>1/2$, then we recall that $g=-h_{xy}$ is a non-negative function with total integral $1$ (by Definition \ref{Definition_class_of_functions}  for the set $\mathfrak F^{\mathrm x,\mathrm y,I}$), so is $-\hat h_{xy}$, and we further have $f_{xy}=h_{xy}-\hat h_{xy}$. Therefore, using \eqref{eq_mu_def},
   $$
    \int_0^1\int_0^1 |\nu(s,t)|\, du dv\le  \frac{1}{c} \int_0^1\int_0^1 |f_{xy}(s,t)|\,ds \, dt\le \frac{2}{c}.
   $$
   Hence, the part of the integral \eqref{eq_x18} with $|\nu(x,y)|>1/2$ is upper-bounded by
   $$\frac{2C}{c}\iint   |\nu| \mathbf 1_{|\nu|>1/2} \, dx\, dy.$$
   \item If $|\nu(s,t)|>1/2$ and $|\nu(x,y)|\le 1/2$, then we upper bound by the part of the integral where only $|\nu(s,t)|>1/2$ (dropping the second condition). Using exactly the same argument is in the previous case, we again get a bound $\frac{2C}{c}\iint   |\nu| \mathbf 1_{|\nu|>1/2} \, dx\, dy$. \qedhere
 \end{itemize}
\end{proof}

\begin{lemma} \label{Lemma_EL_vanish} Suppose that $f(x,y)$ is a difference of two functions from $\mathfrak F^{\mathrm x,\mathrm y,I}$, such that mixed partials of each of them are (minus) densities of absolutely continuous measures on $[0,1]^2$. Also suppose that $\hat h \in \mathfrak F^{\mathrm x,\mathrm y,I}$ is a smooth solution to EL equations of class $[c,C]$. Then
\begin{equation}
\label{eq_EL_integral_form}
 \int_0^1\int_0^1 \left[f_{xy}  \ln\bigl(-\hat h_{xy}\bigr)+ f_{xy}-r \hat h_x f_y - r f_x \hat h_y \right]\, dx\, dy =0.
\end{equation}
\end{lemma}
\begin{proof}
 In terms of $\hat g = - \hat h_{xy}$, the left-hand side of \eqref{eq_EL_integral_form} is rewritten as
 \begin{multline}
 \label{eq_x20}
  \int_0^1\int_0^1 \Biggl[f_{xy}  \ln \frac{\hat g(x,y)\hat g(0,0)}{\hat g(x,0)\hat g(0,y)}+r  f_y \int_0^y \hat g(x,v) \, dv + r f_x \int_0^x \hat g(u,y) \, du\\- r f_x \int_0^1 \hat g(u,y) \, du  +f_{xy}  \left[\ln \hat g(x,0)+\ln \hat g(0,y)-\ln \hat g(0,0)+1\right] \Biggr]\, dx\, dy
 \end{multline}
The integral of the second line vanishes, because $f$ has zero boundary conditions along all sides of $[0,1]^2$. For the first line, we integrate by parts in $x$ for the second term and in $y$ for the third term, transforming \eqref{eq_x20} into
 \begin{multline*}
  \int_0^1\int_0^1 \Biggl[f_{xy}  \ln \frac{\hat g(x,y)\hat g(0,0)}{\hat g(x,0)\hat g(0,y)}-r  f_{xy} \int_0^x \int_0^y \hat g(s,t)\, dt\, ds- r f_{xy} \int_0^y \int_0^x \hat g(s,t) \, ds \, dt \Biggr]\, dx\, dy
  \\+ \int_0^1\Biggl[r  f_{y} \int_0^x \int_0^y \hat g(s,t) \, dt \, ds \Biggr]^{x=1}_{x=0} dy + \int_0^1 \Biggl[r f_{x} \int_0^y \int_0^x \hat g(s,t) \, ds \, dt \Biggr]^{y=1}_{y=0}\,dx.
 \end{multline*}
 The first line of the last expression is zero because of \eqref{eq_rectangle_condition_2}, and the second one vanishes because of the zero boundary conditions for $f$.
\end{proof}

\section{Smooth solutions to Euler-Lagrange equations}\label{Section_smooth}

In this section we discuss how to solve the Euler-Lagrange equations for the maximizers of the permuton energy $\mathfrak E (h)$.

\subsection{Consequences of Definition \ref{Definition_smooth_solution}}

\begin{proposition} \label{Proposition_two_functions}
 Take $r\ne 0$ and let $h$ and $g$ be as in Definition \ref{Definition_smooth_solution}. For each $1\le u\le k$, $1\le v \le \ell$, there exist continuous functions $\phi^{uv}(x):[\mathrm x_{u-1}, \mathrm x_u]\to \mathbb R\cup \{\infty\}$ and $\psi^{uv}(y):[\mathrm y_{v-1}, \mathrm y_v]\to \mathbb R\cup \{\infty\}$, taking the value $\infty$ at most once each, and such that
 \begin{equation}
 \label{eq_Ansatz}
  g^{uv}(x,y)=- \frac{1}{r} \cdot \frac{\frac{\partial}{\partial x} \phi^{uv}(x) \frac{\partial}{\partial y} \psi^{uv}(y)}{[\phi^{uv}(x)-\psi^{uv}(y)]^2}.
 \end{equation}
\end{proposition}
\begin{remark} \label{Remark_Liouville}
A way to arrive at \eqref{eq_Ansatz} is by sending $(x_2-x_1)\to 0$, $(y_2-y_1)\to 0$ in \eqref{eq_rectangle_condition_2}, leading to the (hyperbolic) Liouville PDE:
\begin{equation}
 (\ln g)_{xy}=2r g,
\end{equation}
because \eqref{eq_Ansatz} is one possible form of a general solution to this PDE.
\end{remark}
\begin{remark} \label{Remark_Moebius}
 The choice of $\phi$ and $\psi$ in \eqref{eq_Ansatz} is not unique: the Moebius transformations
 \begin{equation}
   \phi^{uv}(x)\mapsto \frac{\alpha \phi^{uv}(x)+\beta}{\gamma \phi^{uv}(x)+\delta}, \qquad \psi^{uv}(y)\mapsto \frac{\alpha \psi^{uv}(y)+\beta}{\gamma \psi^{uv}(y)+\delta}
 \end{equation}
 keep the expression in the right-hand side of \eqref{eq_Ansatz} unchanged. In particular, these transformations with real $\alpha,\beta,\gamma,\delta$ can be used to move around the points (if any) where $\phi^{uv}(x)=\infty$ or $\psi^{uv}(y)=\infty$. One can also get rid of the infinities completely by using a Moebius transform mapping the real line to the unit circle in $\mathbb C$. Then we can rewrite the representation \eqref{eq_Ansatz} in terms of two angles on the circle as:
 \begin{align}
 \label{eq_Ansatz_trigonometric}
  g^{uv}(x,y)&=\frac{\frac{\partial}{\partial x} \theta_1^{uv}(x) \frac{\partial}{\partial y} \theta_2^{uv}(y)}{r} \cdot \frac{[-\sin \theta_1^{uv}(x)+\ii \cos \theta_1^{uv}(x)][-\sin \theta_2^{uv}(y)+\ii \cos \theta_2^{uv}(y)]}{[\cos \theta_1^{uv}(x)+\ii \sin \theta_1^{uv}(x)-\cos \theta_2^{uv}(y)-\ii \sin \theta_2^{uv}(y)]^2}
  \\& =\frac{1}{r} \cdot \frac{\frac{\partial}{\partial x} \theta_1^{uv}(x) \frac{\partial}{\partial y} \theta_2^{uv}(y)}{[\sin \theta_1^{uv}(x)-\sin \theta_2^{uv}(y)]^2+[\cos \theta_1^{uv}(x)-\cos \theta_2^{uv}(y)]^2}. \notag
 \end{align}
\end{remark}
\begin{proof}[Proof of Proposition \ref{Proposition_two_functions}] Our argument is somewhat similar to \cite[Section 5]{starr2009thermodynamic}.
We fix $1\le u \le k$ and $1\le v\le \ell$ and omit them from the notations throughout the proof. We assume $r<0$ throughout the proof: the case $r>0$ can be obtained by applying $g(x,y)\leftrightarrow g(1-x,y)$ involution. Using \eqref{eq_rectangle_condition_2} inside the rectangle $(\mathrm x_{u-1},\mathrm x_u)\times (\mathrm y_{v-1},\mathrm y_v)$ and extending it to the boundaries by continuity, we conclude that for all $\mathrm x_{u-1}\le x \le \mathrm x_{u}$, $\mathrm y_{v-1}\le y \le \mathrm y_{v}$
\begin{equation}
 \label{eq_rectangle_condition_3}
  g(x,y) \frac{g(\mathrm x_{u-1},\mathrm y_{v-1})}{g(\mathrm x_{u-1},y) g(x,\mathrm y_{v-1})} =\exp\left(2r \int_{\mathrm y_{v-1}}^{y}\int_{\mathrm x_{u-1}}^{x} g(s,t)\, \dd s\, \dd t\right).
\end{equation}
We treat \eqref{eq_rectangle_condition_3} as a functional equation on the unknown function $g$, supplied with boundary conditions $g(x,\mathrm y_{v-1})$, $\mathrm x_{u-1}\le x \le \mathrm x_{u}$, and $g(\mathrm x_{u-1},y)$,  $\mathrm y_{v-1}\le y \le \mathrm y_{v}$. We further argue in two steps. In Step 1, we show that there exists a solution to \eqref{eq_rectangle_condition_3} of the form \eqref{eq_Ansatz}. In Step 2 we show that \eqref{eq_rectangle_condition_3} uniquely determines $g$.

{\bf Step 1.} Note that the right-hand side of \eqref{eq_Ansatz} can be rewritten as
$$
 -\frac{1}{r}\frac{\partial^2}{\partial x \partial y}\ln\bigl|\phi(x)-\psi(y)\bigr|.
$$
Plugging into \eqref{eq_rectangle_condition_3}, we need to check
$$
  \frac{[\phi(x)-\psi(\mathrm y_{v-1})]^2[\phi(\mathrm x_{u-1})-\psi(y)]^2}{[\phi(\mathrm x_{u-1})-\psi(\mathrm y_{v-1})]^2[\phi(x)-\psi(y)]^2} \stackrel{?}{=}\exp\left(-2 \ln \left|\frac{\bigl(\phi(x)-\psi(y)\bigr)\bigl(\phi(x)-\psi( y_{v-1})\bigr)}{\bigl(\phi(x)-\psi( y_{v-1})\bigr)\bigl(\phi(\mathrm x_{u-1})-\psi(y)\bigr)}\right|\right),
$$
which is trivially true for any choice of functions such that $|\phi(x)-\psi(y)|\ne 0$ for all relevant $x$ and $y$. We also need to match the functions to the boundary conditions $g(x,\mathrm y_{v-1})$ and $g(\mathrm x_{u-1},y)$. For instance, we can choose
\begin{align}
\label{eq_x29} \phi(x)&=-\frac{1}{2} +\left[1-\sqrt{|r|} \int_{\mathrm x_{u-1}}^{x} \frac{g(s,\mathrm y_{v-1})}{\sqrt{g(\mathrm x_{u-1},\mathrm y_{v-1})}} \dd s\right]^{-1},\\
 \psi(y)&=\frac{1}{2} - \left[1-\mathrm{sgn}(r)\sqrt{|r|} \int_{\mathrm y_{v-1}}^{y} \frac{g(\mathrm x_{u-1},t)}{\sqrt{g(\mathrm x_{u-1},\mathrm y_{v-1})}} \dd t\right]^{-1}, \label{eq_x23}
\end{align}
leading through \eqref{eq_Ansatz} to
\begin{equation}
\label{eq_x21}
 g(x,y)=\frac{ g(x,\mathrm y_{v-1}) g(\mathrm x_{u-1},y)}{g(\mathrm x_{u-1},\mathrm y_{v-1})}  \left[1-\frac{r}{g(\mathrm x_{u-1},\mathrm y_{v-1})}  \int_{\mathrm x_{u-1}}^{x} g(s,\mathrm y_{v-1}) \dd s \int_{\mathrm y_{v-1}}^{y} g(\mathrm x_{u-1},t) \dd t \right]^{-2},
\end{equation}
which clearly satisfies the desired boundary conditions $g(x,\mathrm y_{v-1})$ and $g(\mathrm x_{u-1},y)$. Note that since $r<0$, \eqref{eq_x21} is well-defined and uniformly bounded for all  $\mathrm x_{u-1}\le x \le \mathrm x_{u}$,  $\mathrm y_{v-1}\le y \le \mathrm y_{v}$. Further, $|\phi(x)|>1/2$ for $\mathrm x_{u-1}<x\le \mathrm x_u$ and $|\psi(y)|<1/2$ for $\mathrm y_{v-1}<y\le\mathrm y_v$, hence $\phi(x)-\psi(y)$ stays bounded away from zero, as desired. Also note that $\phi(x)$ might have a pole, but it cancels when plugging into \eqref{eq_Ansatz}. Clearly, \eqref{eq_x29}, \eqref{eq_x23} imply that $\phi$ and $\psi$ can not have more than one pole, as claimed in the statement.

{\bf Step 2.} Suppose that $g^{[1]}$ and $g^{[2]}$ are two positive continuous functions on $[\mathrm x_{u-1},\mathrm x_{u}]\times[\mathrm y_{v-1},\mathrm y_{v}]$, which both satisfy \eqref{eq_rectangle_condition_3} and $g^{[1]}(x,\mathrm y_{v-1})=g^{[2]}(x,\mathrm y_{v-1})$, $g^{[1]}(\mathrm x_{u-1},y)=g^{[2]}(\mathrm x_{u-1},y)$. We claim that then $g^{[1]}$ and $g^{[2]}$ coincide. Indeed, define
$$
 \Delta g(x,y)=\max_{\begin{smallmatrix} s\in [\mathrm x_{u-1},x]\\ t\in [\mathrm y_{v-1},y]\end{smallmatrix}} |g^{[1]}(s,t)-g^{[2]}(s,t)|.
$$
Subtracting \eqref{eq_rectangle_condition_3} for $g^{[1]}$ and $g^{[2]}$ and noting that $\frac{g(\mathrm x_{u-1},\mathrm y_{v-1})}{g(\mathrm x_{u-1},y) g(x,\mathrm y_{v-1})}$ is the same for both, we deduce existence of a constant $C_1>0$, such that
\begin{equation}
\label{eq_x22}
 \Delta g(x,y)\le C_1 \int_{\mathrm x_{u-1}}^x \int_{\mathrm y_{v-1}}^y \Delta g(s,t)\, \dd t \dd s.
\end{equation}
Plugging \eqref{eq_x22} into itself $(n-1)$ times, and after that bounding $|\Delta g(s,t)|<C_2$, we deduce
$$
 \Delta g(x,y)\le C_2 \frac{\bigl[C_1 (x-\mathrm x_{u-1}) (y-\mathrm y_{v-1})\bigr]^{n}}{n!}.
$$
Since $n=1,2,\dots$ is arbitrary, $\Delta g(x,y)$ has to be $0$.
\end{proof}

The next step is to explain how functions $g^{uv}(x,y)$ are related to each other.

\begin{proposition} \label{Proposition_jumps}
 Let $h$ and $g$ be as in Definition \ref{Definition_smooth_solution}. If $1\le u\le k$ and $1\le v_1 <v_2\le \ell$ are such that $I_{u,v_1+1}=I_{u,v_1+2}=\dots=I_{u,v_2-1}=0$, but $I_{u,v_1}=I_{u,v_2}=1$, then there exists a constant $C_{u,v_1\to v_2}>0$, such that
 \begin{equation}
 \label{eq_vertical_jump}
    g^{u v_2}(x,\mathrm y_{v_2-1})= C_{u,v_1\to v_2}\cdot g^{u v_1}(x,\mathrm y_{v_1}), \qquad \text{for all }x\in [\mathrm x_{u-1},\mathrm x_{u}].
 \end{equation}
 Similarly, if $1\le u_1<u_2\le k$ and $1\le v \le \ell$ are such that $I_{u_1+1,v}=I_{u_1+2,v}=\dots=I_{u_2-1,v}=0$, but $I_{u_1,v}=I_{u_2,v}=1$, then  there exists a constant $C_{u_1\to u_1,v}$, such that
 \begin{equation}
  \label{eq_horizontal_jump}
   g^{u_2 v}(\mathrm x_{u_2-1},y)= C_{u_1\to u_2,v}\cdot g^{u_1 v}(\mathrm x_{u_1},y)\qquad \text{for all }y\in [\mathrm y_{v-1},\mathrm y_{v}].
 \end{equation}
\end{proposition}
In words, Proposition \ref{Proposition_jumps} says that $g(x,y)$ makes a multiplicative jump by a constant when moving to an adjacent rectangle or over a hole in the domain.
\begin{proof}[Proof of Proposition \ref{Proposition_jumps}] For \eqref{eq_vertical_jump}, we choose any $x_1<x_2\in [\mathrm x_{u-1},\mathrm x_{u}]$ and note that \eqref{eq_rectangle_condition_2} for $y_1=y_{v_1}$, $y_2=y_{v_2-1}$ has the right-hand side $1$ and can be rewritten as
$$
 \frac{g^{u v_1}(x_1,\mathrm y_{v_1})}{g^{u v_2}(x_1,\mathrm y_{v_2-1})}=\frac{g^{u v_1}(x_2,\mathrm y_{v_1})}{g^{u v_2}(x_2,\mathrm y_{v_2-1})}.
$$
Hence, the ratio $\frac{g^{u v_1}(x,\mathrm y_{v_1})}{g^{u v_2}(x,\mathrm y_{v_2-1})}$ does not depend on $x$. For \eqref{eq_horizontal_jump} the argument is similar.
\end{proof}
\begin{proposition} \label{Proposition_vertex_relation}
 Let $h$ and $g$ be as in Definition \ref{Definition_smooth_solution}. If $1\le u_1<u_1\le k$ and $1\le v_1<v_2\le \ell$ are such that $I_{u_1v_1}=I_{u_1v_2}=I_{u_2v_1}=I_{u_2v_2}=1$, but $I_{uv}=0$ for all $u_1<u<u_2$ $v_1<v<v_2$, then
 \begin{equation}
  \frac{ g^{u_1 v_1}(\mathrm x_{u_1},\mathrm y_{v_1}) g^{u_2 v_2}(\mathrm x_{u_2-1},\mathrm y_{v_2-1})}{g^{u_1 v_2}(\mathrm x_{u_1},\mathrm y_{v_2-1}) g^{u_2 v_1}(\mathrm x_{u_2-1},\mathrm y_{v_1})}=1.
 \end{equation}
\end{proposition}
\begin{proof}
 We apply \eqref{eq_rectangle_condition_2} for $x_1=\mathrm x_{u_1}$, $x_2=\mathrm x_{u_2-1}$, $y_1=\mathrm y_{v_1}$, $y_2=\mathrm y_{v_2-1}$.
\end{proof}
\begin{proposition} \label{Propositions_all_functions_Moebius}
 The functions $\phi^{uv}$ in Proposition \ref{Proposition_two_functions} are Moebius transforms of each other and the functions $\psi^{uv}$ are also Moebius transforms of each other. In other words there exist functions $\phi:[0,1]\to\mathbb R$ and $\psi:[0,1]\to\mathbb R$, such that for each $u,v$, we have
 $$
  \phi^{uv}(x)=\frac{\alpha_1^{uv} \phi(x)+\beta_1^{uv}}{\gamma_1^{uv} \phi(x)+\delta_1^{uv}},\quad x\in [\mathrm x_{u-1},\mathrm x_{u}]; \qquad  \psi^{uv}(y)=\frac{\alpha_2^{uv} \psi(y)+\beta_2^{uv}}{\gamma_2^{uv} \psi(y)+\delta_2^{uv}}, \quad y\in [\mathrm y_{v-1},\mathrm y_{v}].
 $$
\end{proposition}
\begin{proof}
 Note that for different values of $u$ the functions $\phi^{uv}$ have different domains of definitions (which might only overlap by a point), hence, the statement of being a Moebius transform of each other is trivially true. Hence, we only fix $u$ and  study what happens as we vary $v$. Note that only those $v$ where $I_{uv}=1$ matter, and take two adjacent such $v$s, i.e.\ suppose $v_1<v_2$ are as in Proposition \ref{Proposition_jumps}. Using \eqref{eq_vertical_jump} and substituting \eqref{eq_Ansatz} for $g^{u v_1}$, $g^{u v_2}$, we get:
 \begin{equation}
 \label{eq_x25}
    \frac{\partial}{\partial x} \left[ \frac{\frac{\partial}{\partial y} \psi^{uv_2}(\mathrm y_{v_2-1}) }{\phi^{u v_2}(x) - \psi^{u v_2}(\mathrm y_{v_2-1})}\right]= \frac{\partial}{\partial x} \left[\frac{C_{u,v_1\to v_2} \frac{\partial}{\partial y} \psi^{uv_1}(\mathrm y_{v_1}) }{\phi^{u v_1}(x)-\psi^{u v_1}(\mathrm y_{v_1})}\right]
 \end{equation}
 Hence, for a constants $C_1$, $C_2$, $C_3$, $C_4$, and $C_5$, not depending on $x$, we have
 $$
   \frac{C_1}{\phi^{u v_2}(x) - C_2}= \frac{C_3}{\phi^{u v_1}(x)-C_4}+C_5.
 $$
 Therefore, $\phi^{u v_2}(x)$ and $\phi^{u v_1}(x)$ are Moebius transforms of each other. The argument for $\psi$ is the same.
\end{proof}

For a general matrix $I$, all the functions $\phi^{uv}$ and $\psi^{uv}$ need to be related to each other using Propositions \ref{Proposition_jumps}, \ref{Proposition_vertex_relation}, and \ref{Propositions_all_functions_Moebius}. An example will be given in Section \ref{Section_non_convex_example}. For convex $I$ there is a simplification.

\begin{proposition}
 Suppose that $I$ is convex. Then all functions $\phi^{uv}$ and $\psi^{uv}$ in Proposition \ref{Proposition_two_functions} can be chosen to not depend on $u$ or $v$. The resulting $\phi$ and $\psi$ can be chosen to be continuous functions from $[0,1]$ to the extended real line $\mathbb R\cup\{\infty\}$ with finitely many $\infty$ values.
\end{proposition}
\begin{proof}
 For each $v=1,2,\dots,\ell$, we let $\{u_-(v),u_-(v)+1,\dots,u_+(v)\}$ denote the maximal segment of indices such that $I_{uv}=1$. Consider the rectangle $[\mathrm x_{u_-(v)-1},\mathrm x_{u_+(v)}]\times [\mathrm y_{v-1},\mathrm y_{v}]$. Repeating the argument of Proposition \ref{Proposition_two_functions}, we can choose continuous functions $\phi^v$ and $\psi^v$, such that
 \begin{equation}
 \label{eq_x24}
   g^{uv}(x,y)=- \frac{1}{r} \cdot \frac{\frac{\partial}{\partial x} \phi^{v}(x) \frac{\partial}{\partial y} \psi^{v}(y)}{[\phi^{v}(x)-\psi^{v}(y)]^2}, \qquad u\in \{u_-(v), u_-(v)+1,\dots, u_+(v)\}.
 \end{equation}
 Indeed, the argument of Proposition \ref{Proposition_two_functions} did not use the continuity of $g$, and it works equally well for
 piecewise continuous $g$.

 Combining Propositions \ref{Proposition_jumps} and \ref{Proposition_vertex_relation}, we conclude that as we cross the horizontal line $y=\mathrm y_v$, the function $g(x,y)$ is being multiplied by a single constant $C_{v\to v+1}$. Hence, for each $v=1,2,\dots,\ell-1$ we have a version of \eqref{eq_x25}, which now reads:
  \begin{equation}
 \label{eq_x26}
    \frac{\partial}{\partial x} \left[ \frac{\frac{\partial}{\partial y} \psi^{v+1}(\mathrm y_{v}) }{\phi^{v+1}(x) - \psi^{v+1}(\mathrm y_{v})}\right]= \frac{\partial}{\partial x} \left[\frac{C_{v\to v+1} \frac{\partial}{\partial y} \psi^{v}(\mathrm y_{v}) }{\phi^{ v}(x)-\psi^{ v}(\mathrm y_{v})}\right], \qquad \text{ or }
 \end{equation}
  \begin{equation}
 \label{eq_x27}
   \frac{\frac{\partial}{\partial y} \psi^{v+1}(\mathrm y_{v}) }{\phi^{v+1}(x) - \psi^{v+1}(\mathrm y_{v})}+ \tilde C_{v\to v+1}= \frac{C_{v\to v+1} \frac{\partial}{\partial y} \psi^{v}(\mathrm y_{v}) }{\phi^{ v}(x)-\psi^{ v}(\mathrm y_{v})} , \qquad \text{ or }
 \end{equation}
  \begin{equation}
 \label{eq_x28}
    C_{v\to v+1} \frac{\partial}{\partial y} \psi^{v}(\mathrm y_{v})\left[\frac{\frac{\partial}{\partial y} \psi^{v+1}(\mathrm y_{v}) }{\phi^{v+1}(x) - \psi^{v+1}(\mathrm y_{v})}+ \tilde C_{v\to v+1}\right]^{-1}+\psi^{ v}(\mathrm y_{v})= \phi^{ v}(x).
 \end{equation}
We now use \eqref{eq_x28} to sequentially redefine the functions $\phi^v$ and $\psi^v$ for $v=2,3,\dots,\ell$. Functions
$\phi^1$ and $\psi^1$ are unchanged. We replace $(\phi^{2}, \psi^2)$ with $(\tilde \phi^2, \tilde \psi^2)$ defined by:
  \begin{align*}
  \tilde \phi^{2}(x)&:= C_{1\to2} \frac{\partial}{\partial y} \psi^{1}(\mathrm y_{1})\left[\frac{\frac{\partial}{\partial y} \psi^{2}(\mathrm y_{1}) }{\phi^{2}(x) - \psi^{2}(\mathrm y_{1})}+ \tilde C_{1\to 2}\right]^{-1}+\psi^{1}(\mathrm y_{1}),\\
   \tilde \psi^{2}(y)&:= C_{1\to2} \frac{\partial}{\partial y} \psi^{1}(\mathrm y_{1})\left[\frac{\frac{\partial}{\partial y} \psi^{2}(\mathrm y_{1}) }{\psi^{2}(y) - \psi^{2}(\mathrm y_{1})}+ \tilde C_{1\to 2}\right]^{-1}+\psi^{1}(\mathrm y_{1}).
 \end{align*}
As we noticed in Remark \ref{Remark_Moebius}, the simultaneous Moebius transformations of $\phi$ and $\psi$ do not change $g$. Hence, the redefined $(\phi^2, \psi^2)$ still satisfy \eqref{eq_x24}. On the other hand, comparing with \eqref{eq_x28}, we see that $\tilde \phi^2(x)=\phi^1(x)$ for all $x$ for which both functions are simultaneously defined. Repeating this procedure for all larger $v=3,4,\dots,\ell$, we achieve that all $\phi$ functions are the same. All $\psi$ functions have different domains of definitions, hence, they can also be thought as being independent of $v$.

We also need to check the continuity. Note that $\phi^v$ were continuous (taking values in extended real line $\mathbb R\cup\{\infty\}$ by construction, because \eqref{eq_x29} is a continuous function of $x$; the continuity is preserved under Moebius transformations, hence, this the final $\phi(x)$ is continuous in $x$.

For the $\psi$ function we need a separate argument, as it might have been discontinuous at $y=\mathrm y_v$ for some $v$. In order to see that this does not happen, we go back to \eqref{eq_x27}. Taking into account that after redefining $\phi$, it no longer depends on $v$, we have
  \begin{equation}
 \label{eq_x30}
   \frac{\frac{\partial}{\partial y} \psi^{v+1}(\mathrm y_{v}) }{\phi(x) - \psi^{v+1}(\mathrm y_{v})}+ \tilde C_{v\to v+1}= \frac{C_{v\to v+1} \frac{\partial}{\partial y} \psi^{v}(\mathrm y_{v}) }{\phi(x)-\psi^{ v}(\mathrm y_{v})}.
 \end{equation}
 Treating everything except $\phi(x)$ as constants, \eqref{eq_x30} is an identity of two functions of $x$. Clearly, this identity can be true only if $\psi^{v+1}(\mathrm y_{v})=\psi^{ v}(\mathrm y_{v})$, which proves the desired continuity.
\end{proof}
\begin{remark}
 In several examples we checked with non-convex $I$, it was impossible to find a solution of the EL-equations of the form \eqref{eq_Ansatz} with $\phi^{uv}$ and $\psi^{uv}$ not depending on $u$ or $v$.
\end{remark}

\subsection{Algebraic equations for convex $I$}
\label{Section_Convex_equations}

Summarizing the developments of the previous section, in order to solve the EL-equations for the permuton energy with convex $I$ we need to find two continuous functions
$$
 \phi:[0,1]\mapsto \mathbb R\cup\{\infty\}, \qquad \psi:[0,1]\mapsto \mathbb R\cup\{\infty\},
$$
such that the function
$$
 g(x,y)=- \frac{1}{r} \cdot \frac{\frac{\partial}{\partial x} \phi(x) \frac{\partial}{\partial y} \psi(y)}{[\phi(x)-\psi(y)]^2}
$$
is non-negative whenever $(x,y)\in (\mathrm x_{u-1},\mathrm x_u)\times(\mathrm y_{v-1},\mathrm y_v)$ with $I_{uv}=1$ and satisfies the permuton conditions (equivalent to the boundry conditions on $h$):
\begin{equation}
\label{eq_vertical_integral}
 \sum_{v=1}^\ell I_{uv} \int_{\mathrm y_{v-1}}^{\mathrm y_v} g(x,y)\dd y = 1 ,\qquad x \in (\mathrm x_{u-1},\mathrm x_u), \quad 1\le u \le k,
\end{equation}
\begin{equation}
\label{eq_horizontal_integral}
 \sum_{u=1}^k I_{uv} \int_{\mathrm x_{u-1}}^{\mathrm x_u} g(x,y)\dd x = 1 ,\qquad y \in (\mathrm y_{v-1},\mathrm y_v), \quad 1\le v \le \ell.
\end{equation}

The algorithm for finding functions $\phi$ and $\psi$ is given in the following theorem.
\begin{theorem} \label{Theorem_convex_solution}
 The $k+\ell+2$ numbers $\phi(\mathrm x_{u})$, $0\le u \le k$, $\psi(\mathrm y_v)$, $0\le v\le \ell$, are found by solving a system of $k+\ell$ polynomial equations:
 \begin{align}
 \label{eq_poly_1}e^{r(\mathrm x_u-\mathrm x_{u-1})}&= \prod_{v=1}^\ell  \left[ \frac{\psi(\mathrm y_v)-\phi(\mathrm x_{u-1}) }{\psi(\mathrm y_{v-1})-\phi(\mathrm x_{u-1})}\cdot  \frac{\psi(\mathrm y_{v-1})-\phi(\mathrm x_{u})}{\psi(\mathrm y_v)-\phi(\mathrm x_u)}\right]^{I_{uv}}, \quad 1\le u \le k,\\ \label{eq_poly_2}
 e^{r(\mathrm y_v-\mathrm y_{v-1})}&= \prod_{u=1}^k \left[ \frac{\psi(\mathrm y_v)-\phi(\mathrm x_{u-1}) }{\psi(\mathrm y_{v-1})-\phi(\mathrm x_{u-1})}\cdot  \frac{\psi(\mathrm y_{v-1})-\phi(\mathrm x_{u})}{\psi(\mathrm y_v)-\phi(\mathrm x_u)}\right]^{I_{uv}}, \quad 1\le v \le \ell.
\end{align}
 The values at other points are found from polynomial equations in variables $\phi(x)$ and $\psi(x)$:
\begin{align}
\label{eq_poly_3}
 e^{r(x-\mathrm x_{u-1})}&= \prod_{v=1}^\ell \left[\frac{\psi(\mathrm y_{v-1})- \phi(x)  }{\psi(\mathrm y_v)- \phi(x)} \cdot  \frac{\psi(\mathrm y_v)- \phi(\mathrm x_{u-1})} {\psi(\mathrm y_{v-1})- \phi(\mathrm x_{u-1}) } \right]^{I_{uv}},\qquad  x\in [\mathrm x_{u-1},\mathrm x_u], \quad 1\le u \le k,\\
 e^{r(y-\mathrm y_{v-1})}&= \prod_{u=1}^k \left[\frac{\psi(y)-\phi(\mathrm x_{u-1})  }{\psi(y)-\phi(\mathrm x_u) } \cdot  \frac{\psi(\mathrm y_{v-1})-\phi(\mathrm x_u) } {\psi(\mathrm y_{v-1})- \phi(\mathrm x_{u-1})} \right]^{I_{uv}},\qquad y\in [\mathrm y_{v-1},\mathrm y_v], \quad 1\le v \le \ell. \notag
\end{align}
\end{theorem}
\begin{remark}
  Recall from Remark \ref{Remark_Moebius} that functions $\phi$ and $\psi$ are defined up to $3$--dimensional group of Moebius transformations. Hence, we can fix three out of $k+\ell+2$ variables $\phi(\mathrm x_{u})$,  $\psi(\mathrm y_v)$ in arbitrary ways, and find the remaining $k+\ell-1$ variables from \eqref{eq_poly_1},\eqref{eq_poly_2}. One equation is abundant: the product of all $k$ equations \eqref{eq_poly_1} is the same as the produc of all $\ell$ equations \eqref{eq_poly_2}. Therefore, eventually the number of variables matches the number of equations.
\end{remark}
\begin{proof}[Proof of Theorem \ref{Theorem_convex_solution}]
 We rewrite $g(x,y)=-\frac{1}{r}\frac{\partial^2}{\partial x \partial y} \ln\bigl(\phi(x)-\psi(y)\bigr)$ and compute
 \begin{equation}
 \label{eq_x48}
   \int_{\mathrm y_{v-1}}^{\mathrm y_v} \int_{\mathrm x_{u-1}}^{\mathrm x_u} g(x,y)\,\dd x \dd y=-\frac{1}{r}\ln\left[ \frac{(\phi(\mathrm x_u)-\psi(\mathrm y_v))(\phi(\mathrm x_{u-1})-\psi(\mathrm y_{v-1}))}{(\phi(\mathrm x_{u-1})-\psi(\mathrm y_v))(\phi(\mathrm x_u)-\psi(\mathrm y_{v-1}))}\right].
 \end{equation}
 Hence, integrating \eqref{eq_vertical_integral} from $\mathrm x_{u-1}$ to $\mathrm x_u$, we get
 $$
  r (\mathrm x_u-\mathrm x_{u-1})= - \sum_{v=1}^{\ell} I_{uv} \ln\left[ \frac{(\phi(\mathrm x_u)-\psi(\mathrm y_v))(\phi(\mathrm x_{u-1})-\psi(\mathrm y_{v-1}))}{(\phi(\mathrm x_{u-1})-\psi(\mathrm y_v))(\phi(\mathrm x_u)-\psi(\mathrm y_{v-1}))}\right].
 $$
 Exponentiating, we get \eqref{eq_poly_1}. Similarly, integrating \eqref{eq_horizontal_integral} from $\mathrm y_{v-1}$ to $\mathrm y_v$ and exponentiating, we get \eqref{eq_poly_2}.

 By the same trick, the first equation of \eqref{eq_poly_3} is obtained by integrating \eqref{eq_vertical_integral} from $\mathrm x_{u-1}$ and then exponentiating. The second equation is obtained by integrating \eqref{eq_horizontal_integral} from $\mathrm y_{v-1}$ to $y$ and then exponentiating.
\end{proof}

Theorem \ref{Theorem_convex_solution} has the following corollary.
\begin{corollary}\label{Corollary_g_form}
For $I$ convex, the functions $\phi(x)$ and $\psi(y)$ are Moebius transformations of $e^{rx}$ and $e^{ry}$, respectively, on each rectangle.
Consequently $g(x,y)$ has the form
\be\label{grationalform}g^{uv}(x,y) = -\frac{(ad-bc)re^{r(x+y)}}{(a+be^{ry}+ce^{rx}+de^{r(x+y)})^2}\ee
on each rectangle, where the constants $a,b,c,d$ depend on $u,v$.
\end{corollary}

\begin{proof} Because $I$ is convex, the equations (\ref{eq_poly_3}) telescope, leading to, for each $u$,
$$e^{r(x-x_{u-1})} = C_u \frac{A_u-\phi(x)}{B_u-\phi(x)}$$ for some constants $A_u,B_u, C_u$, and similarly for $\psi(y)$ for each $v$. This implies the first statement.
Now applying (\ref{eq_Ansatz}) yields the general form (\ref{grationalform}) of $g$.
\end{proof}

Theorem \ref{Theorem_convex_solution} indicates that finding the limit
shapes for restricted random permutations for convex $I$ crucially depends on being able to solve the system of equations \eqref{eq_poly_1}, \eqref{eq_poly_2}. In Sections \ref{Section_r0}, \ref{Section_numeric_alg}, \ref{Section_unique_solution} we explain various approaches for finding the solutions of \eqref{eq_poly_1}, \eqref{eq_poly_2}, and then show additional examples in Section \ref{Section_examples}.

\medskip

When $I$ is non-convex, the algorithm is more complicated. According to Proposition \ref{Propositions_all_functions_Moebius}, the functions $\phi^{uv}$ and $\psi^{uv}$ now depend on $u$ and $v$ through Moebius transforms and all these transforms should be identified by using the relationships of Proposition \ref{Proposition_jumps} and \ref{Proposition_vertex_relation}; in principle this can be done (one can check that the number of degrees of freedom matches the number of constraints), yet computationally it is quite hard. After this is done, one can again use formula \eqref{eq_poly_1}, \eqref{eq_poly_2}, \eqref{eq_poly_3}, but with $\phi$ and $\psi$ replaced with $\phi^{uv}$ and $\psi^{uv}$, respectively, in the factor raised to the $I_{uv}$ power. Through this approach, we managed to compute the simplest non-convex example in Section \ref{Section_non_convex_example}. It would be interesting to find further simplifications, which would allow the investigation of other non-convex situations.

\subsection{The case $r=0$}
\label{Section_r0}
In the case $r=0$ the equations \eqref{eq_poly_1}, \eqref{eq_poly_2}, and \eqref{eq_poly_3} are dramatically simplified.
\begin{theorem} \label{Theorem_r_0}
 For $r=0$, the maximizer in the variational problem  $\mathfrak E(h)\to \max$, $h\in \mathfrak F^{\mathrm x,\mathrm y,I}$, has the following form in terms of density $g=-h_{xy}$:
 \begin{itemize}
  \item On each rectangle $(\mathrm x_{u-1},\mathrm x_u)\times(\mathrm y_{v-1},\mathrm y_v)$, the density $g(x,y)$ is constant:
   \begin{equation} \label{eq_g_constant} g(x,y)=\frac{g^{uv}}{( \mathrm x_{u}-\mathrm x_{u-1})(\mathrm y_v-\mathrm y_{v-1})}, \qquad x\in (\mathrm x_{u-1},\mathrm x_u), \quad y\in  (\mathrm y_{v-1},\mathrm y_v).
   \end{equation}
  \item There exist two sequences of positive reals $\lambda_1,\lambda_2,\dots,\lambda_k$ and $\mu_1,\mu_2,\dots,\mu_\ell$, such that $g^{uv}=I_{uv} \lambda_u \mu_v$. These sequences are found from the $k+\ell$ conditions:
  \begin{align}
  \label{eq_stochasticity_condition_1}
   \mathrm x_{u}-\mathrm x_{u-1}&=\sum_{v=1}^{\ell} I_{uv} \lambda_u \mu_v,\qquad u=1,2,\dots,k,\\
   \mathrm y_v-\mathrm y_{v-1}&=\sum_{u=1}^{k} I_{uv} \lambda_u \mu_v,\qquad v=1,2,\dots,\ell. \label{eq_stochasticity_condition_2}
  \end{align}

 \end{itemize}
\end{theorem}
\begin{proof} We explain why two sequences $\lambda$ and $\mu$ satisfying \eqref{eq_stochasticity_condition_1},\eqref{eq_stochasticity_condition_2} exist immediately after the proof. Taking this as granted, note that, given \eqref{eq_g_constant}, the conditions \eqref{eq_stochasticity_condition_1},\eqref{eq_stochasticity_condition_2} are equivalent to the boundary conditions on $h$ or \eqref{eq_vertical_integral}, \eqref{eq_horizontal_integral}. Simultaneously, the condition \eqref{eq_rectangle_condition_2} with $r=0$ holds, because $g(x,y)$ factors as a function of $x$ times a function of $y$. Therefore, $g(x,y)$ of \eqref{eq_g_constant} satisfies the Euler-Lagrange equations and, hence, by Theorem \ref{Theorem_maximizer_small_r} it is a maximizer.
\end{proof}

In matrix notation, \eqref{eq_stochasticity_condition_1},\eqref{eq_stochasticity_condition_2}, ask to multiply $I_{uv}$ on the right and left by two diagonal matrices, so that its row and column sums become equal to the prescribed constants (e.g.\ if these constants were all $1$, then we would have been aiming for a bistochastic matrix). This problem is well-studied in the statistics, computer science, economics, and optimal transport literature; the algorithm for solving it was rediscovered many times: some of the names are \emph{iterative proportional fitting} and \emph{Sinkhorn–Knopp algorithm}, see \cite{idel2016review} for the detailed historic overview and many references.

This is an iterative algorithm, which starts by setting  $\mu^{[0]}=(1,\dots,1)$ and then sequentially updating $\lambda^{[n]}$, $\mu^{[n]}$, $n=1,2,3,\dots$ by the following two rules:
\begin{enumerate}
 \item Using $\mu^{[n-1]}$, we define $\lambda^{[n]}$ so that \eqref{eq_stochasticity_condition_1} holds, i.e.:
 $$
  \lambda^{[n]}_u=\frac{\mathrm x_u-\mathrm x_{u-1}}{\sum_{v=1}^\ell I_{uv} \mu^{[n-1]}_v},\qquad 1\le u \le k.
 $$
 \item Using $\lambda^{[n]}$, we define $\mu^{[n]}$ so that \eqref{eq_stochasticity_condition_2} holds, i.e.:
 $$
  \mu^{[n]}_v=\frac{\mathrm y_v-\mathrm y_{v-1}}{\sum_{u=1}^k I_{uv} \lambda^{[n]}_u},\qquad 1\le v \le \ell.
 $$
\end{enumerate}
Assume that $I$ and $\{\mathrm x_u\}_{u=1}^k$, $\{\mathrm y_v\}_{v=1}^{\ell}$ are non-degenerate, as in Definition \ref{Definition_non_degenerate}. Then $\lambda^{[n]}$, $\mu^{[n]}$ of the above iterative algorithm converge as $n\to\infty$ towards $\lambda$ and $\mu$ solving \eqref{eq_stochasticity_condition_1}, \eqref{eq_stochasticity_condition_2}, see \cite[Theorem 4.1]{idel2016review} and references therein. In practice, the convergence is extremely fast, as many authors show.

\subsection{Numeric algorithms for small $r$}\label{Section_numeric_alg}

We next explore how to leverage the existence of the iterative proportional fitting algorithm at $r=0$, in order to obtain numeric solutions to the polynomial equations \eqref{eq_poly_1}, \eqref{eq_poly_2} for small $r$.

We first connect two sets of equations: \eqref{eq_poly_1}, \eqref{eq_poly_2} and \eqref{eq_stochasticity_condition_1}, \eqref{eq_stochasticity_condition_2}.

\begin{lemma} \label{Lemma_r_to_0}
 Make a change of variables $\phi(\mathrm x_u)=\frac{1}{r \chi(\mathrm x_u)}$, take the logarithm of \eqref{eq_poly_1}, \eqref{eq_poly_2}, and divide by $r$. Sending $r\to 0$, one arrives at the equations \eqref{eq_stochasticity_condition_1}, \eqref{eq_stochasticity_condition_2} under identification
 \begin{equation}
   \chi(\mathrm x_{u})-\chi(\mathrm x_{u-1})=\lambda_u, \quad 1\le u \le k,\qquad \psi(\mathrm y_{v})-\psi(\mathrm y_{v-1})=\mu_v, \quad 1\le v \le \ell.
 \end{equation}
\end{lemma}
\begin{proof} Once we make the change of variables, we have as $r\to 0$
\begin{multline*}
\ln\left[ \frac{\psi(\mathrm y_v)-\phi(\mathrm x_{u-1}) }{\psi(\mathrm y_{v-1})-\phi(\mathrm x_{u-1})}\cdot  \frac{\psi(\mathrm y_{v-1})-\phi(\mathrm x_{u})}{\psi(\mathrm y_v)-\phi(\mathrm x_u)}\right]=
 \ln\left[ \frac{1-r\chi(\mathrm x_{u-1})\psi(\mathrm y_v) }{1-r\chi(\mathrm x_{u-1})\psi(\mathrm y_{v-1})}\cdot  \frac{1-r\chi(\mathrm x_{u})\psi(\mathrm y_{v-1})}{1-r\chi(\mathrm x_u)\psi(\mathrm y_v)}\right]
 \\ =
 r \left[(\chi(\mathrm x_u)-\chi(\mathrm x_{u-1}))(\psi(\mathrm y_v)-\psi(\mathrm y_{v-1}))\right] +o(r). \qedhere
\end{multline*}

\end{proof}

Combining the $r=0$ algorithm of the previous section with Lemma \ref{Lemma_r_to_0} we arrive at the following numerical method for solving \eqref{eq_poly_1}, \eqref{eq_poly_2}.

\begin{itemize}
 \item Encode the $r$--dependent solution through
 $$
  \phi(\mathrm x_u)=\frac{1}{r \chi^r(\mathrm x_u)}, \quad   \psi(\mathrm y_v)=\psi^r(\mathrm y_v).
 $$
 \item At $r=0$, find $\chi^{0}(\mathrm x_u)$, $0\le u \le k$, and $\psi^{0}(\mathrm y_v)$, $0\le v\le \ell$, through the iterative proportional fitting algorithm of the
 previous section.
 \item Make steps $r\to r+\Delta r$ with very small $\Delta r$, by writing
  $$
   \chi^{r+\Delta r}(\mathrm x_u)=\chi^r(\mathrm x_u)+\Delta \chi^r(\mathrm x_u), \qquad   \psi^{r+\Delta r}(\mathrm y_v)=\psi^r(\mathrm y_v)+\Delta\psi^r(\mathrm y_v),
  $$
  assuming that the increments are small, linearizing the equations \eqref{eq_poly_1}, \eqref{eq_poly_2} in terms of the increments, and solving
  the resulting system of linear equations.
\end{itemize}

One delicate point is that, as we mentioned in Remark \ref{Remark_Moebius}, for each $r$, there is a $3$--dimensional group of symmetries of the solutions, and we need to
choose one solution, by manually imposing three conditions (e.g.\ choosing the values of three variables). Similarly, at $r=0$, the $3$--dimensional group
is generated by shifting $\chi^0$ by a constant, shifting $\psi^0$ by a constant, and simultaneous multiplication of $\chi^0$ by a constant together with
division of $\chi^0$ by the same constant. It is unclear which choice of the three conditions is the best from the numerics perspective. We experimented
with some choices and implemented several versions of the above algorithm --- this is how our Figure \ref{Figure_44simulations} is produced.
In all our experiments, the algorithms work well for small $r$ and start exploding at some point as $r$ grows (by growing to $\infty$ or hitting one of the zeros of denominators in the
equations \eqref{eq_poly_1}, \eqref{eq_poly_2}).

\begin{figure}[t]
\begin{center}
   \includegraphics[width=0.4\linewidth]{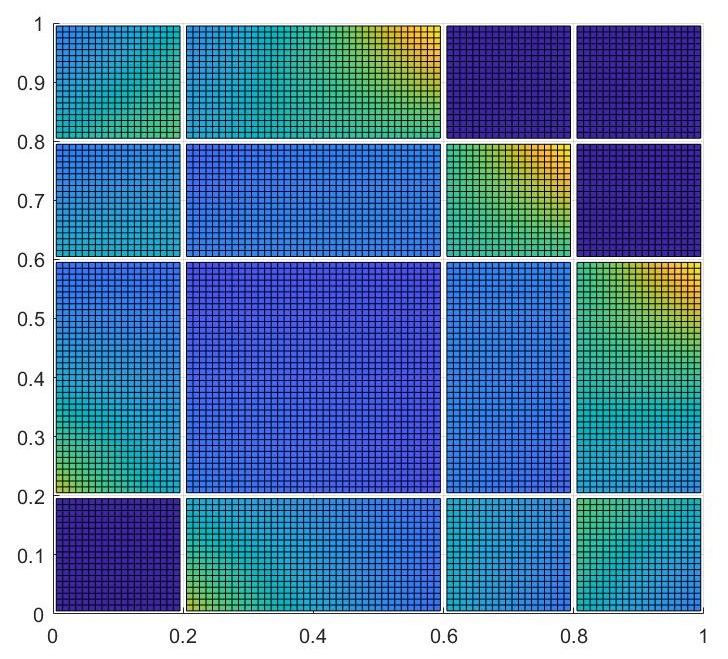} \hfill \includegraphics[width=0.55\linewidth]{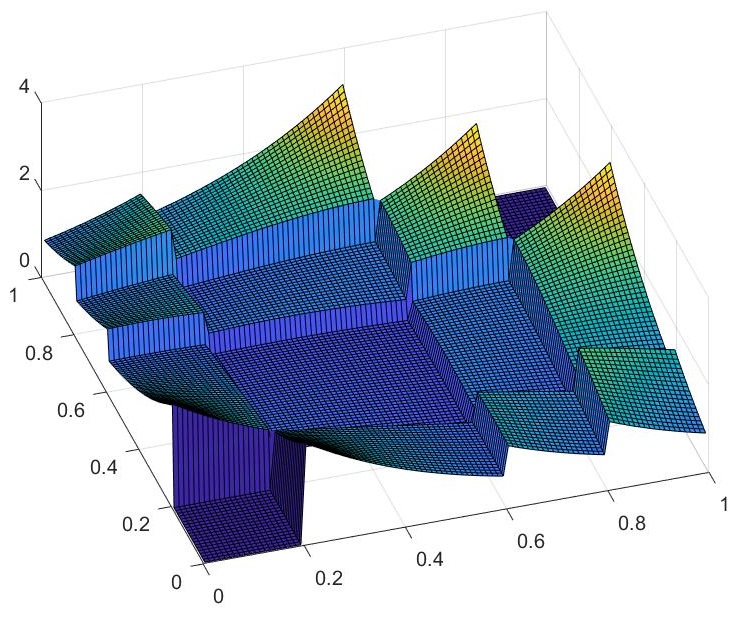}

   \includegraphics[width=0.4\linewidth]{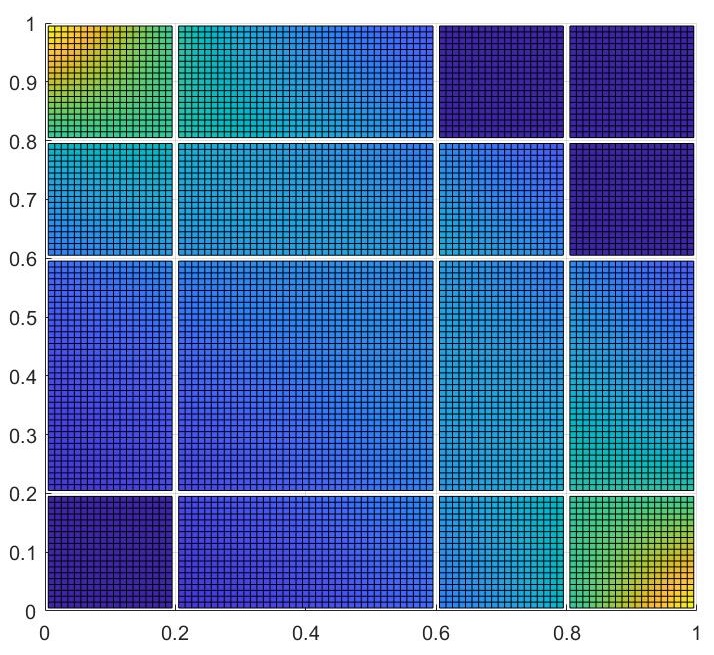} \hfill \includegraphics[width=0.55\linewidth]{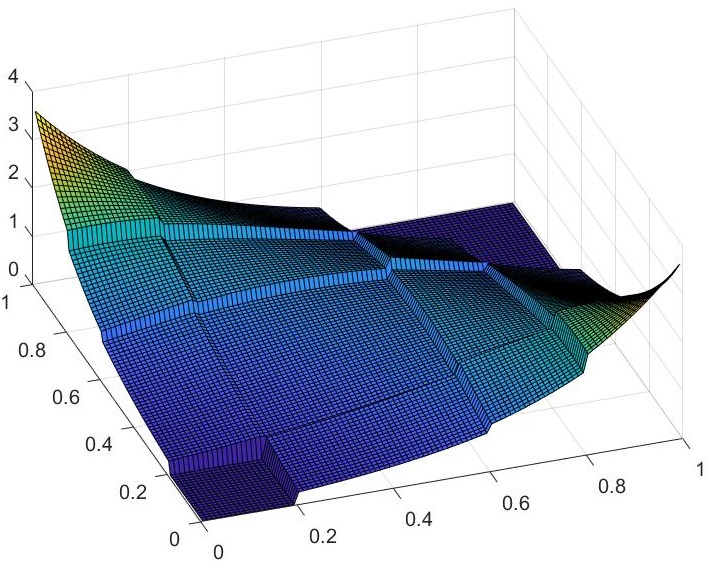}
\end{center}
        \caption{\label{Figure_44simulations} Numeric solution of the EL equations for $k=4$, $(\mathrm x_u)=(\mathrm y_v)=(0, \tfrac{1}{5},\tfrac{3}{5},\tfrac{4}{5},1)$, $I=\protect\begin{pmatrix} 1& 1 & 0 & 0\\ 1 & 1 & 1 & 0\\ 1& 1 & 1 & 1 \\ 0 & 1 & 1 & 1 \protect\end{pmatrix}$, $r=3$ in the top panel and $r=-3$ in the bottom panel. Pictures obtained by refining to $100\times 100$ grid,
        and solving trigonometric equations \eqref{eq_poly_1_tr}, \eqref{eq_poly_2_tr} by increasing $r$ by small steps from $r=0$. }
\end{figure}

It would be interesting to implement a similar numeric algorithm for the case of a general, not necessarily convex matrix $I$. For the $r=0$ case,
Theorem \ref{Theorem_r_0} did not require convexity of $I$; hence, it can still serve as a starting point.

\begin{remark}
 Continuing Remark \ref{Remark_Moebius}, one can transform the functions $\phi$ and $\psi$ into the unit circle and reparameterize
 them by angles
 $\theta$ and $\rho$. Equations \eqref{eq_poly_1}, \eqref{eq_poly_2} then turn into an alternative trigonometric form:
  \begin{align}
 \label{eq_poly_1_tr}e^{r(\mathrm x_u-\mathrm x_{u-1})}&=
 \prod_{v=1}^\ell  \left[ \frac{\cos\bigl(\sqrt{r}(\theta_v+\rho_{u-1})\bigr) \cos\bigl(\sqrt{r}(\theta_{v-1}+\rho_{u})\bigr)}{\cos\bigl(\sqrt{r}(\theta_v+\rho_{u})\bigr) \cos\bigl(\sqrt{r}(\theta_{v-1}+\rho_{u-1})\bigr) }\right]^{I_{uv}}, \quad 1\le u \le k,\\ \label{eq_poly_2_tr}
 e^{r(\mathrm y_v-\mathrm y_{v-1})}&= \prod_{u=1}^k  \left[ \frac{\cos\bigl(\sqrt{r}(\theta_v+\rho_{u-1})\bigr) \cos\bigl(\sqrt{r}(\theta_{v-1}+\rho_{u})\bigr)}{\cos\bigl(\sqrt{r}(\theta_v+\rho_{u})\bigr) \cos\bigl(\sqrt{r}(\theta_{v-1}+\rho_{u-1})\bigr) }\right]^{I_{uv}}, \quad 1\le v \le \ell.
\end{align}
where we rescaled the variables, so that the $r\to 0$ limit (after taking logarithms) becomes transparent using $\cos(x)=1-\frac{x^2}{2}+o(x^3)$, $x\to 0$.
\end{remark}

\subsection{Domains with unique solution to algebraic equations}
\label{Section_unique_solution}

In general, \eqref{eq_poly_1},\eqref{eq_poly_2} can be polynomial equations of a high degree and might have  many solutions.
One can hope that in many situations imposing the positivity conditions following from $g(x,y)\ge 0$ leads to a unique choice for the solution, but we do not have precise theorems in this direction, see
Section \ref{Section_examples} for some examples.
Instead, in this section we describe a class of arrays $I$, for which \eqref{eq_poly_1}, \eqref{eq_poly_2} have a unique solution modulo Moebius transformations,
and this solution can be found by a straightforward
exact algorithm.

The idea is to reduce arrays $I_{uv}$ to smaller arrays, until we reach $1\times 1$ array $I_{11}$, for which solving \eqref{eq_poly_1}, \eqref{eq_poly_2} is immediate. The reductions are based on two operations, which take a convex $k\times \ell$ array $I$ and output a convex array $\tilde I$ of smaller size:

\begin{enumerate}
 \item If $I$ has a unique $1$ in the first row or the last row, then remove this row and output the remaining $k\times (\ell-1)$ array. If $I$ has a unique $1$ in the first column or the last column, then remove this column and output the remaining $(k-1)\times \ell$ array.
 \item If two rows of $I$, $v$ and $v+1$ are completely filled with $1$s: $I_{uv}=I_{u,v+1}=1$ for all $1\le u \le k$, then remove the row $v$ and output the remaining $k\times (\ell-1)$ array. If two columns of $I$, $u$ and $u+1$ are completely filled with $1$s: $I_{uv}=I_{u+1,v}=1$ for all $1\le v \le \ell$, then remove the column $u+1$ and output the remaining $(k-1)\times \ell$ array.
\end{enumerate}

\begin{definition}
 A convex (in the sense of Definition \ref{Definition_convex_array}) $k\times \ell$ array $I$ is called \emph{simple}, if it can be reduced by a sequence of the above operations to the $1\times 1$ array $I_{11}=1$.
\end{definition}
For $2\times 2$ case, all possible arrays (with at least one $1$ in each row and column) are both convex and simple. For $2\times 3$ case, there are several non-convex arrays, but all convex arrays are simple. For $3\times 3$ case some convex arrays are simple and some are not. For example,
$\begin{pmatrix} 0 & 0 & 1 \\ 1 & 1 & 1 \\ 1 & 1 &  0\end{pmatrix}$ and $\begin{pmatrix} 1 & 1 & 1 \\ 1 & 1 & 1 \\ 1 & 1 &  0\end{pmatrix}$ are simple, while $\begin{pmatrix} 0 & 1 & 1 \\ 1 & 1 & 1 \\ 1 & 1 &  0\end{pmatrix}$ is not.

\begin{theorem}\label{Theorem_unique_case}
 Fix $r\ne 0$, a simple convex $k\times \ell$ array $I$, and any sequences of numbers $\mathrm x_0<\mathrm x_1<\dots<\mathrm x_k$, $\mathrm y_0<\mathrm y_1<\dots<\mathrm y_\ell$, such that $\mathrm x_k-\mathrm x_0=\mathrm y_\ell-\mathrm y_0$, and the data is non-degenerate in the sense of Definition \ref{Definition_non_degenerate}. Then the solution $\mathrm \phi(\mathrm x_u)$, $0 \le u \le k$, $\mathrm \psi(\mathrm y_v))$, $0\le v \le \ell$, to equations \eqref{eq_poly_1}, \eqref{eq_poly_2} is unique up to Moebius transformations\footnote{I.e.\ any two solutions can be obtained from each other by  $\phi(x)\mapsto \frac{\alpha \phi(x)+\beta}{\gamma \phi(x)+\delta}$, $\psi(y)\mapsto \frac{\alpha \psi(y)+\beta}{\gamma \psi(y)+\delta}$.} and can be obtained by the recursive algorithm in the following proof.
\end{theorem}
\begin{proof} The argument is by induction in $k+\ell$. For the base $k=\ell=1$, both equations \eqref{eq_poly_1}, \eqref{eq_poly_2} are the same and read
\begin{equation}
\label{eq_x31}
  e^{r(\mathrm x_1-\mathrm x_{0})}=   \frac{\psi(\mathrm y_1)-\phi(\mathrm x_{0}) }{\psi(\mathrm y_{0})-\phi(\mathrm x_{0})}\cdot  \frac{\psi(\mathrm y_{0})-\phi(\mathrm x_{1})}{\psi(\mathrm y_1)-\phi(\mathrm x_1)}.
\end{equation}
Note that the four variables  involved, $\phi(\mathrm x_0)$, $\phi(\mathrm x_1)$, $\psi(\mathrm y_0)$, $\psi(\mathrm y_1)$, should be all distinct in order for \eqref{eq_x31} to hold. We can apply a Moebius transformation to map three of them to any three numbers in $\mathbb R\cup \{\infty\}$; since the cross ratio is invariant under Moebius transformations, \eqref{eq_x31} is not going to change. Say, let us map $(\phi(\mathrm x_0), \psi(\mathrm y_0),\psi(\mathrm y_1))\mapsto (0,\infty,1)$. Then the equation turns into
$$
  e^{r(\mathrm x_1-\mathrm x_{0})}=   \frac{1}{1-\phi(\mathrm x_1)},
$$
which clearly has a unique solution $\phi(\mathrm x_1)\in\mathbb R$.

For the induction step, we explore the two operations in the definition of the simple array. For the first operation, assume without loss of generality that $I$ has a unique non-zero element in the last column, at position $I_{k \hat v}=1$ and $\tilde I$ is obtained from $I$ by removing that column.

The last, $u=k$, equation \eqref{eq_poly_1} reads
\begin{equation}
\label{eq_x32}
  e^{r(\mathrm x_k-\mathrm x_{k-1})}=   \frac{\psi(\mathrm y_{\hat u})-\phi(\mathrm x_{k-1}) }{\psi(\mathrm y_{\hat u-1})-\phi(\mathrm x_{k-1})}\cdot  \frac{\psi(\mathrm y_{\hat u-1})-\phi(\mathrm x_{k})}{\psi(\mathrm y_{\hat u})-\phi(\mathrm x_k)}.
\end{equation}
We see that once $\phi(\mathrm x_{k-1})$, $\psi(\mathrm y_{\hat u})$, $\psi(\mathrm y_{\hat u-1})$ are determined, there is a unique $\phi(\mathrm x_k)$ satisfying \eqref{eq_x32}. On the other hand, the $k+\ell+1$ variables $\mathrm \phi(\mathrm x_u)$, $0 \le u \le k-1$, $\mathrm \psi(\mathrm y_v))$, $0\le v \le \ell$, satisfy a system of $k+\ell-1$ equations:
 \begin{align*}
 %\label{eq_x33}
 e^{r(\mathrm x_u-\mathrm x_{u-1})}&= \prod_{v=1}^\ell  \left[ \frac{\psi(\mathrm y_v)-\phi(\mathrm x_{u-1}) }{\psi(\mathrm y_{v-1})-\phi(\mathrm x_{u-1})}\cdot  \frac{\psi(\mathrm y_{v-1})-\phi(\mathrm x_{u})}{\psi(\mathrm y_v)-\phi(\mathrm x_u)}\right]^{I_{uv}}, \quad 1\le u \le k-1,\\ % \label{eq_x34}
 e^{r(\mathrm y_v-\mathrm y_{v-1})}&= \prod_{u=1}^{k-1} \left[ \frac{\psi(\mathrm y_v)-\phi(\mathrm x_{u-1}) }{\psi(\mathrm y_{v-1})-\phi(\mathrm x_{u-1})}\cdot  \frac{\psi(\mathrm y_{v-1})-\phi(\mathrm x_{u})}{\psi(\mathrm y_v)-\phi(\mathrm x_u)}\right]^{I_{uv}}, \quad 1\le v \le \ell, \quad v\ne \hat v,\\
 % \label{eq_x35}
 e^{r(\mathrm y_{\hat v}-\mathrm y_{\hat v-1})}  e^{-r(\mathrm x_k-\mathrm x_{k-1})}&= \prod_{u=1}^{k-1} \left[ \frac{\psi(\mathrm y_{\hat v})-\phi(\mathrm x_{u-1}) }{\psi(\mathrm y_{\hat v-1})-\phi(\mathrm x_{u-1})}\cdot  \frac{\psi(\mathrm y_{\hat v-1})-\phi(\mathrm x_{u})}{\psi(\mathrm y_{\hat v})-\phi(\mathrm x_u)}\right]^{I_{u\hat v}},
\end{align*}
where we used \eqref{eq_x32} to get the last equation. The last system of equations is again of the form \eqref{eq_poly_1},\eqref{eq_poly_2} for the $(k-1)\times \ell$ array $\tilde I$, and for the difference $y_{\hat v}-\mathrm y_{\hat v-1}$ replaced with $(y_{\hat v}-\mathrm y_{\hat v-1})-(\mathrm x_k-\mathrm x_{k-1})$. (Notice that $(y_{\hat v}-\mathrm y_{\hat v-1})-(\mathrm x_k-\mathrm x_{k-1})> 0$, because of the non-degeneracy of the data, as in Definition \ref{Definition_non_degenerate}.) Hence, by the induction assumption, this system has a unique (up to Moebius transformations) solution, and then supplementing with \eqref{eq_x32} so does the original system \eqref{eq_poly_1}, \eqref{eq_poly_2}.
\bigskip

For the second operation, assume without loss of generality that the columns $\hat u$ and $\hat u+1$ of $I$ are completely filled with 1s and let $\tilde I$ be obtained from $I$ by removing the column $\hat u+1$. The equations \eqref{eq_poly_1} for $u=\hat u$ and $u=\hat u+1$ read after cancellations:
\begin{align*}
 e^{r(\mathrm x_{\hat u}-\mathrm x_{\hat u-1})}&= \frac{\psi(\mathrm y_\ell)-\phi(\mathrm x_{\hat u-1}) }{\psi(\mathrm y_{0})-\phi(\mathrm x_{\hat u-1})}\cdot  \frac{\psi(\mathrm y_{0})-\phi(\mathrm x_{\hat u})}{\psi(\mathrm y_\ell)-\phi(\mathrm x_{\hat u})},\\
  e^{r(\mathrm x_{\hat u+1}-\mathrm x_{\hat u})}&= \frac{\psi(\mathrm y_\ell)-\phi(\mathrm x_{\hat u}) }{\psi(\mathrm y_{0})-\phi(\mathrm x_{\hat u})}\cdot  \frac{\psi(\mathrm y_{0})-\phi(\mathrm x_{\hat u+1})}{\psi(\mathrm y_\ell)-\phi(\mathrm x_{\hat u+1})}.
\end{align*}
Let us replace the last two equations by the first one and their product, i.e. by
\begin{align}
\label{eq_x33}
 e^{r(\mathrm x_{\hat u}-\mathrm x_{\hat u-1})}&= \frac{\psi(\mathrm y_\ell)-\phi(\mathrm x_{\hat u-1}) }{\psi(\mathrm y_{0})-\phi(\mathrm x_{\hat u-1})}\cdot  \frac{\psi(\mathrm y_{0})-\phi(\mathrm x_{\hat u})}{\psi(\mathrm y_\ell)-\phi(\mathrm x_{\hat u})},\\
  e^{r(\mathrm x_{\hat u+1}-\mathrm x_{\hat u-1})}&= \frac{\psi(\mathrm y_\ell)-\phi(\mathrm x_{\hat u-1}) }{\psi(\mathrm y_{0})-\phi(\mathrm x_{\hat u-1})}\cdot  \frac{\psi(\mathrm y_{0})-\phi(\mathrm x_{\hat u+1})}{\psi(\mathrm y_\ell)-\phi(\mathrm x_{\hat u+1})}. \label{eq_x34}
\end{align}
Once $\phi(\mathrm x_{\hat u-1})$, $\psi(\mathrm y_{0})$, $\psi(\mathrm y_{\ell})$ are determined, the equation \eqref{eq_x33} uniquely determines $\phi(\mathrm x_{\hat u})$. On the other hand, $k+\ell+1$ variables $\phi(\mathrm x_0),\dots,\phi(\mathrm x_{\hat u-1}),\phi(\mathrm x_{\hat u+1}),\dots,\phi(\mathrm x_k)$; $\psi(\mathrm y_0),\dots,\psi(\mathrm y_\ell)$ satisfy a system of $k+\ell-1$ equations, given by \eqref{eq_x34} and
 \begin{align}
 \label{eq_x35}
 e^{r(\mathrm x_u-\mathrm x_{u-1})}&= \prod_{v=1}^\ell  \left[ \frac{\psi(\mathrm y_v)-\phi(\mathrm x_{u-1}) }{\psi(\mathrm y_{v-1})-\phi(\mathrm x_{u-1})}\cdot  \frac{\psi(\mathrm y_{v-1})-\phi(\mathrm x_{u})}{\psi(\mathrm y_v)-\phi(\mathrm x_u)}\right]^{I_{uv}}, \quad 1\le u\le \hat u-1 \text{ or } \hat u+2 \le u \le k,\\ \label{eq_x36}
 e^{r(\mathrm y_v-\mathrm y_{v-1})}&= \prod_{u=1}^k \left[ \frac{\psi(\mathrm y_v)-\phi(\mathrm x_{u-1}) }{\psi(\mathrm y_{v-1})-\phi(\mathrm x_{u-1})}\cdot  \frac{\psi(\mathrm y_{v-1})-\phi(\mathrm x_{u})}{\psi(\mathrm y_v)-\phi(\mathrm x_u)}\right]^{I_{uv}}, \quad 1\le v \le \ell.
\end{align}
Note that there is no $\phi(\mathrm x_{\hat u})$ in the last equation, as it cancels between $u=\hat u$ and $u=\hat u+1$ factors. We conclude that the system \eqref{eq_x34},\eqref{eq_x35},\eqref{eq_x36} is exactly of the same form as \eqref{eq_poly_1}, \eqref{eq_poly_2} for $(k-1)\times \ell$ matrix $\tilde I$. Hence, the induction hypothesis applies and this system has a unique solution up to Moebius transformations.
\end{proof}

\medskip

Let us demonstrate how the algorithm works for the case $k=\ell=2$, $\mathrm x=(0,a,1)$, $\mathrm y=(0,b,1)$, $I=\begin{pmatrix} 1 & 0 \\ 1 & 1 \end{pmatrix}$. The inductive reductions are $I\to \begin{pmatrix} 1\\ 1\end{pmatrix}\to \begin{pmatrix} 1 \end{pmatrix}$. On the first step we remove the last column of $I$ and record the equation for determining $\phi(1)$:
\begin{equation}
\label{eq_x37}
  e^{r(1-a)}=   \frac{\psi(b)-\phi(a) }{\psi(0)-\phi(a)}\cdot  \frac{\psi(0)-\phi(1)}{\psi(b)-\phi(1)}.
\end{equation}
On the second step, we remove the top row of $ \begin{pmatrix} 1\\ 1\end{pmatrix}$ array --- we can use either of the two operations for that, say, we use the first one again. We get an equation for determining $\psi(1)$:
\begin{equation}
\label{eq_x38}
  e^{r(1-b)}=   \frac{\psi(1)-\phi(0) }{\psi(b)-\phi(0)}\cdot  \frac{\psi(b)-\phi(a)}{\psi(1)-\phi(a)}.
\end{equation}
The final equation for the $1\times 1$ matrix $I$ is
\begin{equation}
\label{eq_x39}
  e^{r(a+b-1)}=   \frac{\psi(b)-\phi(0) }{\psi(0)-\phi(0)}\cdot  \frac{\psi(0)-\phi(a)}{\psi(b)-\phi(a)}.
\end{equation}
Note that it is necessary to have $a+b>1$ for the data to be non-degenerate in the sense of Definition \ref{Definition_non_degenerate}. Due to Moebius invariance, we can fix three parameters in an arbitrary way and we choose $\phi(0)=0$, $\psi(0)=\infty$, $\psi(b)=1$. Then \eqref{eq_x39} turns into
\begin{equation}
\label{eq_x40}
  e^{r(a+b-1)}=   \frac{1}{1-\phi(a)} \quad \Longleftrightarrow \quad \phi(a)=1-e^{-r(a+b-1)}.
\end{equation}
Through \eqref{eq_x37}, \eqref{eq_x38} we compute:
\begin{equation}
  \phi(1)=1-e^{-rb}, \qquad \psi(1)=\frac{1-e^{-r(a+b-1)}}{1-e^{-ra}}.
\end{equation}
Further, we reconstruct $\phi(x)$ and $\psi(y)$ in all other points through equations \eqref{eq_poly_3}:

\begin{align*}
 e^{r x}= \frac{\frac{1-e^{-r(a+b-1)}}{1-e^{-ra}}}{\frac{1-e^{-r(a+b-1)}}{1-e^{-ra}}- \phi(x)}\quad &\Longleftrightarrow \quad  \phi(x)= \frac{1-e^{-r(a+b-1)}}{1-e^{-ra}}(1-e^{-rx}),\qquad  x\in [0,a],\\
 e^{r(x-a)}= \frac{e^{-r(a+b-1)}}{1- \phi(x)}\quad &\Longleftrightarrow \quad  \phi(x)=1-e^{-r(x+b-1)}, \qquad  x\in [a,1],
 \\
 e^{ry}= \frac{\psi(y)  }{\psi(y)-1+e^{-rb} } \quad &\Longleftrightarrow \quad \psi(y)=\frac{1-e^{-rb}}{1-e^{-ry}},\qquad y\in [0,b],
 \\
  e^{r(y-b)}= \frac{\psi(y)  e^{-r(a+b-1)} }{\psi(y)-1+e^{-r(a+b-1)} } \quad &\Longleftrightarrow \quad \psi(y)=\frac{1-e^{-r(a+b-1)}}{1-e^{-r(a+y-1)}},\qquad y\in [b,1].
\end{align*}
It remains to plug the last formulas into  $g(x,y)=- \frac{1}{r} \cdot \frac{\frac{\partial}{\partial x} \phi(x) \frac{\partial}{\partial y} \psi(y)}{[\phi(x)-\psi(y)]^2}$ for $(x,y)\in (0,a)\times (0,b)\cup (a,1)\times(0,b)\cup (0,a)\times(b,1)$. Figure \ref{Figure_22simulations} plots the result.

\begin{figure}[t]
\begin{center}
   \includegraphics[width=0.48\linewidth]{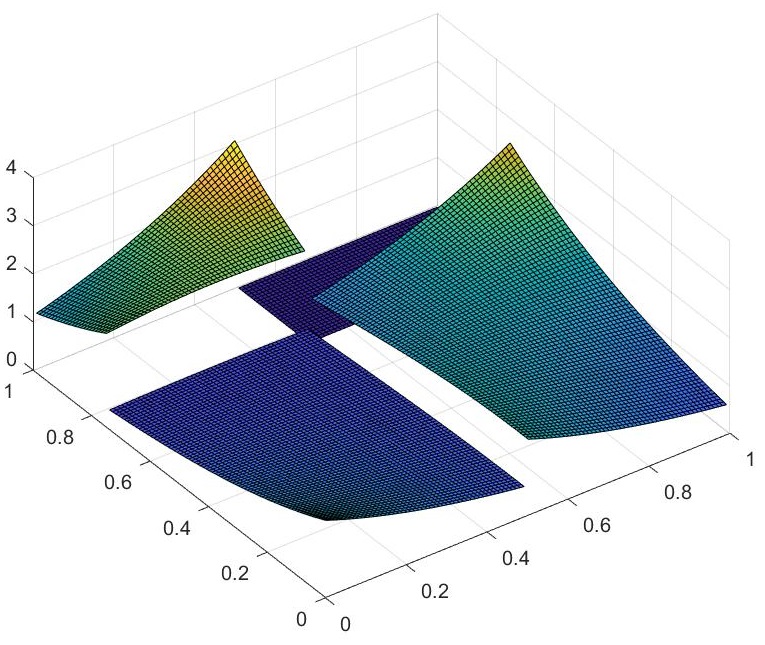} \hfill \includegraphics[width=0.48\linewidth]{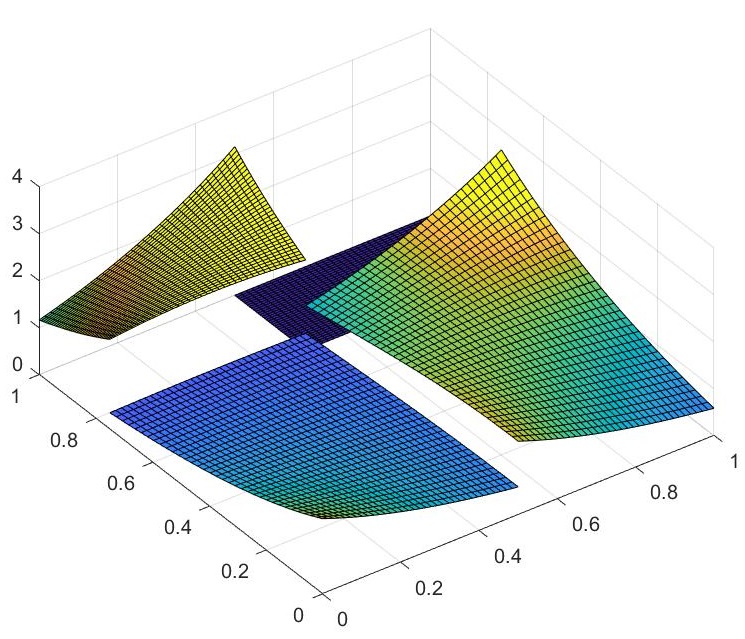}
\end{center}
        \caption{\label{Figure_22simulations}  Solution of EL equations for $I=\protect\begin{pmatrix} 1 & 0\\ 1 & 1\protect\end{pmatrix}$, $\mathrm x=(0,\tfrac{1}{2},1)$, $\mathrm y=(0,\tfrac{3}{4},1)$. Left panel: numeric solution via $100\times 100$ grid refinement. Right panel: exact formulas. }
\end{figure}

%*(x>0)+...*(x<=0)
\section{Further examples}
\label{Section_examples}
\subsection{Convex $3\times 3$ example}
Listing all possible cases, one checks that for $2\times 1$, $1\times 2$, $3\times 1$, $2\times 2$, $1\times 3$, $3\times 2$, and $2\times 3$ arrays, $I$ being convex implies being simple. Hence, we can always use the algorithm of Section \ref{Section_unique_solution} to solve the equations \eqref{eq_poly_1}, \eqref{eq_poly_2}. For $3\times 3$ case, there are two non-simple convex arrays:
$$
 \begin{pmatrix} 0 & 1 & 1\\ 1 & 1 & 1\\ 1 & 1 & 0\end{pmatrix} \text{ and }  \begin{pmatrix} 1 & 1 & 0\\ 1 & 1 & 1\\ 0 & 1 & 1\end{pmatrix}.
$$
In the analysis of \eqref{eq_poly_1},\eqref{eq_poly_2} these two arrays differ by a symmetry accompanied by $r\to -r$ change. Hence, it would be sufficient to only study the first one, which we do in this section.

Let us denote $\mathrm x=(0<a_1<a_2<1)$ and $\mathrm y=(0<b_1<b_2<1)$. We need to determine the values of $8$ variables $\phi(0), \phi(a_1), \phi(a_2),\phi(1)$, $\psi(0),\psi(b_1),\psi(b_2),\psi(1)$, solving $6$ equations \eqref{eq_poly_1},\eqref{eq_poly_2}. The equations involve products of $7$ different cross ratios, corresponding to $7$ pairs $(u,v)$ with $I_{uv}=1$. Let $X$ and $Y$ denote \ the two missing cross ratios, corresponding to $I_{uv}=0$:
\begin{equation}
\label{eq_x41}
 X=\frac{\psi(1)-\phi(0)}{\psi(b_2)-\phi(0)}\cdot \frac{\psi(b_2)-\phi(a_1)}{\psi(1)-\phi(a_1)}, \qquad Y= \frac{\psi(b_1)-\phi(a_2)}{\psi(0)-\phi(a_2)} \cdot \frac{\psi(0)-\phi(1)}{\psi(b_1)-\phi(1)}.
\end{equation}
Noting telescoping cancellations in \eqref{eq_poly_1}, \eqref{eq_poly_2}, we can rewrite them as six equations, and further fix $\phi(0)=0$, $\phi(1)=1$, $\psi(0)=\infty$ by using Moebius invariance, to get
$$
 \begin{cases} e^{r a_1}\cdot X &= \frac{\psi(1)-\phi(0) }{\psi(0)-\phi(0)}\cdot  \frac{\psi(0)-\phi(a_1)}{\psi(1)-\phi(a_1)},\\
 e^{r (a_2-a_1)}&=\frac{\psi(1)-\phi(a_1) }{\psi(0)-\phi(a_1)}\cdot  \frac{\psi(0)-\phi(a_2)}{\psi(1)-\phi(a_2)},\\
 e^{r (1-a_2)}\cdot Y&= \frac{\psi(1)-\phi(a_2) }{\psi(0)-\phi(a_2)}\cdot  \frac{\psi(0)-\phi(1)}{\psi(1)-\phi(1)},
 \\
 e^{r b_1}\cdot Y&=  \frac{\psi(b_1)-\phi(0) }{\psi(0)-\phi(0)}\cdot  \frac{\psi(0)-\phi(1)}{\psi(b_1)-\phi(1)},\\
 e^{r(b_2-b_1)}&=  \frac{\psi(b_2)-\phi(0) }{\psi(b_1)-\phi(0)}\cdot  \frac{\psi(b_1)-\phi(1)}{\psi(b_2)-\phi(1)},\\
 e^{r(1-b_2)}\cdot X&= \frac{\psi(1)-\phi(0) }{\psi(b_2)-\phi(0)}\cdot  \frac{\psi(b_2)-\phi(1)}{\psi(1)-\phi(1)}.
 \end{cases}
 \Longleftrightarrow
  \begin{cases} e^{r a_1}\cdot X &= \frac{\psi(1)}{\psi(1)-\phi(a_1)},\\
 e^{r (a_2-a_1)}&=\frac{\psi(1)-\phi(a_1) }{\psi(1)-\phi(a_2)},\\
 e^{r (1-a_2)}\cdot Y&= \frac{\psi(1)-\phi(a_2) }{\psi(1)-1},
 \\
 e^{r b_1}\cdot Y&=  \frac{\psi(b_1)}{\psi(b_1)-1},\\
 e^{r(b_2-b_1)}&=  \frac{\psi(b_2) }{\psi(b_1)}\cdot  \frac{\psi(b_1)-1}{\psi(b_2)-1},\\
 e^{r(1-b_2)}\cdot X&= \frac{\psi(1) }{\psi(b_2)}\cdot  \frac{\psi(b_2)-1}{\psi(1)-1}.
 \end{cases}
$$
Assuming $X$ and $Y$ to be known, the system of equations is readily solved in terms of them, resulting in:
\begin{equation}\label{eq_x42}
\psi(b_1)=\frac{1}{1- e^{-r b_1}/Y}, \quad \psi(b_2)=\frac{1}{1- e^{-r b_2}/Y}, \quad \psi(1)=\frac{1}{1- e^{-r }/(XY)},
\end{equation}
\begin{equation}\label{eq_x43}
 \phi(a_1)=\frac{1-e^{-r a_1}/ X}{1- e^{-r }/(XY)}, \quad \phi(a_2)=\frac{1-e^{-r a_2}/ X}{1- e^{-r }/(XY)}.
\end{equation}
In order to find $X$ and $Y$, we plug the values of $\psi$ and $\phi$ back into \eqref{eq_x41} and get two equations:
\begin{equation}
\label{eq_x44}
 X=\frac{1-\frac{(1- e^{-r b_2}/Y)(1-e^{-r a_1}/ X)}{1- e^{-r }/(XY)}}{e^{-r a_1}/X}, \qquad Y= \frac{1-\frac{(1-e^{-r a_2}/ X)(1- e^{-r b_1}/Y)}{1- e^{-r }/(XY)}}{e^{-r b_1}/Y}.
\end{equation}
Simplifying using $X\ne 0$, $Y\ne 0$, we get
$$
  e^{-r}- XY= e^{-r(1-a_1)} + e^{-r b_2}-X e^{-r (b_2-a_1)} -Y, \qquad e^{-r}-XY= e^{-r(1-b_1)}+e^{-r a_2} - X - Y e^{-r(a_2-b_1)},
$$
which is equivalent to a linear relation and a quadratic equation:
\begin{equation}
\label{eq_x45}
 Y=\frac{e^{-r(1-a_1)} + e^{-r b_2} -e^{-r(1-b_1)}-e^{-r a_2}+X(1- e^{-r (b_2-a_1)})}{1- e^{-r(a_2-b_1)}},
\end{equation}
\begin{multline}
\label{eq_x46}
  X^2(1- e^{-r (b_2-a_1)})+X(e^{-r(1-a_1)} + e^{-r b_2}+e^{-r(a_2-a_1+b_2- b_1)} -e^{-r(1-b_1)}-e^{-r a_2}-1)\\  +e^{-r(1-b_1)}+e^{-r a_2} + e^{-r(1+a_2-b_1)}- e^{-r(1-a_1+a_2-b_1)} - e^{-r (b_2+a_2-b_1)}-e^{-r}=0.
\end{multline}
The two solutions of the last quadratic equation lead through \eqref{eq_x45} and then \eqref{eq_x42}, \eqref{eq_x43} to two solutions of the system \eqref{eq_poly_1}, \eqref{eq_poly_2}, which then lead to the formula for $g(x,y)$ via \eqref{eq_poly_3}. In order to pick the correct solution we need to additionally impose the positivity condition:
\begin{equation}
 \label{eq_x49}
   -\frac{1}{r}\ln\left[ \frac{(\phi(\mathrm x_u)-\psi(\mathrm y_v))(\phi(\mathrm x_{u-1})-\psi(\mathrm y_{v-1}))}{(\phi(\mathrm x_{u-1})-\psi(\mathrm y_v))(\phi(\mathrm x_u)-\psi(\mathrm y_{v-1}))}\right]>0, \qquad \text{whenever } I_{uv}=1.
 \end{equation}
 which is a corollary of positivity of the density $g(x,y)$ and \eqref{eq_x48}.

The formula \eqref{eq_x46} is quite complicated in the general case and it is intructive to specialize to equally-spaced grid $a_1=b_1=\tfrac{1}{3}$, $a_2=b_2=\tfrac{2}{3}$. The quadratic equation after dividing by $(1- e^{-r/3})$ becomes
\begin{equation}
\label{eq_x47}
  X^2-X(1+e^{-r/3})  +  2e^{-2r/3}-e^{-r} =0,
\end{equation}
and $Y=X$. One directly checks that, when writing two solutions of \eqref{eq_x47} in terms of $\pm \sqrt{\mathrm{Discriminant}}$, the correct sign giving \eqref{eq_x49} is ``$+$'' for $r>0$ and ``$-$'' for $r<0$.

As $r\to 0$, one should rescale $X=1+r X_{[0]}$, with \eqref{eq_x47} turning asymptotically into
\begin{equation}
\label{eq_x50}
  X_{[0]}^2+ \frac{1}{3}X_{[0]}  -\frac{1}{9}=0.
\end{equation}
This time, chosing the correct solution is simple: Theorem \ref{Theorem_r_0} and Lemma \ref{Lemma_r_to_0} imply that we need the positive one.

See two plots of the density for different values of $r$ in Figure \ref{Figure_33simulations}.

\begin{figure}[t]
\begin{center}
   \includegraphics[width=0.48\linewidth]{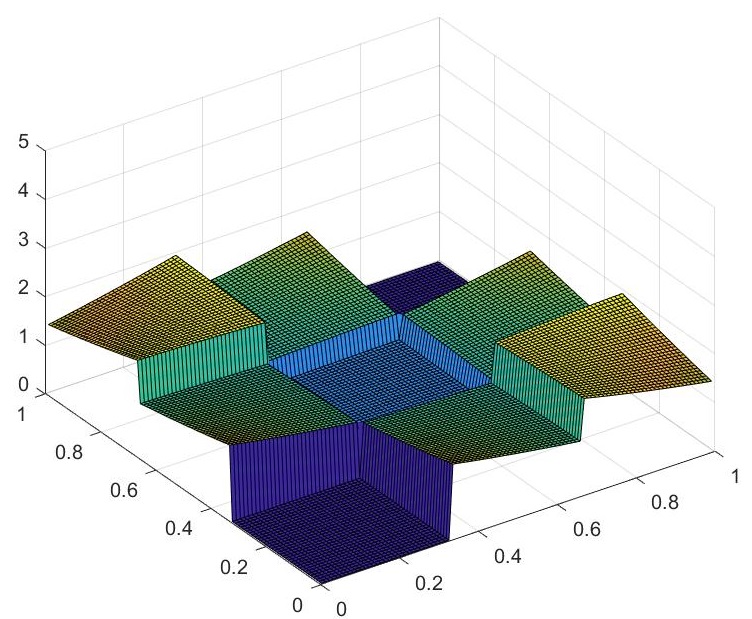} \hfill \includegraphics[width=0.48\linewidth]{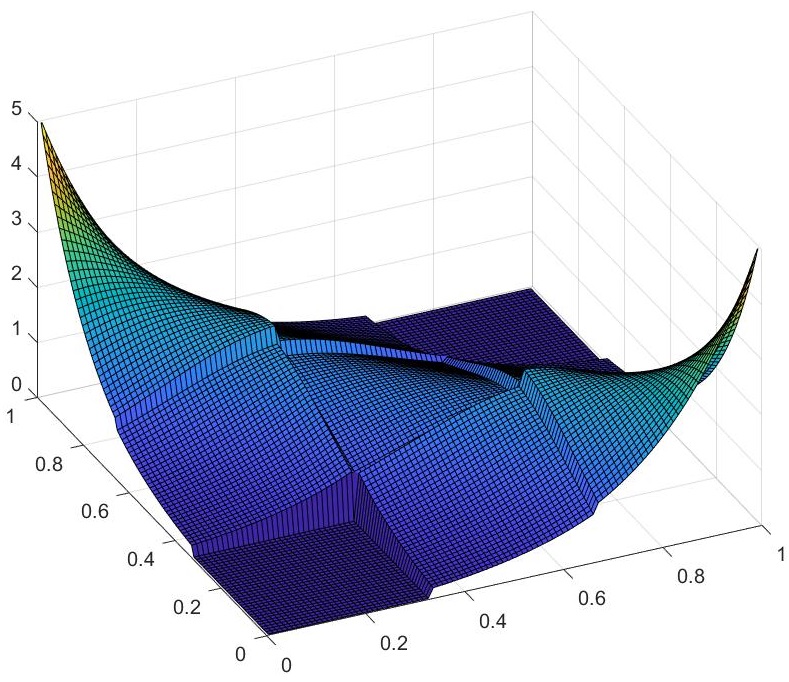}
\end{center}
        \caption{\label{Figure_33simulations}  Solution of EL equations for $I=\protect\begin{pmatrix} 0 &1 & 1\\ 1 & 1 & 1\\ 1 & 1 & 0\protect\end{pmatrix}$, $\mathrm x=\mathrm y=(0,\tfrac{1}{3},\tfrac{2}{3},1)$. Left panel: $r=1$. Right panel: $r=-5$}
\end{figure}

\subsection{Non-convex $3\times 2$ example}
\label{Section_non_convex_example}
 For our next example, we consider a non-convex $I$. We take $k=3$, $\ell=2$, $I=\begin{pmatrix} 1& 0 & 1 \\ 1 & 1 & 1\end{pmatrix}$ and $\mathrm x=(0,\tfrac{1}{3},\tfrac{2}{3},1)$, $\mathrm y=(0,\tfrac{2}{3},1)$. The formulas of Theorem \ref{Theorem_convex_solution} do not apply in this setting, and we need to follow a more complicated path, as outlined at the end of Section \ref{Section_Convex_equations}. By (a small generalization of) Corollary \ref{Corollary_g_form}, we have that $\phi^{uv}$ and $\psi^{uv}$ are Moebius transforms of $e^{rx}$ and $e^{ry}$, and
\be\label{gform}
  g^{uv}(x,y)= -\frac{(a_{uv}d_{uv}-b_{uv}c_{uv})re^{rx + ry}}{(a_{uv}+b_{uv}e^{ry}+ c_{uv}e^{rx}+d_{uv}e^{rx+ry})^2}.
\ee
Each pair $uv$ has $3$ variables ($a,b,c,d$ up to scale), for a total of $15$ constants to determine. They are found from equations given by Proposition \ref{Proposition_jumps}, Proposition \ref{Proposition_vertex_relation}, and versions of \eqref{eq_poly_1}, \eqref{eq_poly_2} expressing the uniform marginals for $g(x,y)$. After putting all those conditions into mathematical software, we found a unique solution: \begin{equation} \label{eq_x51}
g(x, y) =
\begin{cases}
\frac{r e^{r(x + y)} \bigl(1 + e^{-\frac{r}{2}}\bigr) \bigl(1-e^{-\frac{2r}{3}} \bigr) }{\left(1 - e^{ry} - e^{rx} + \bigl[-e^{-\frac{r}{2}} + e^{-\frac{2r}{3}} + e^{-\frac{7r}{6}}\bigr]e^{r(x+y)}\right)^2}, & 0 < x < \frac{1}{3} \text{ and } 0 < y < \frac{2}{3},
\\
\frac{r e^{r(x + y)}\bigl(1- e^{-\frac{r}{3}}\bigr) \bigl(e^{-\frac{r}{6}} + e^{-\frac{2r}{3}}\bigr)}{\left(1 +\bigl[- e^{-\frac{r}{6}} + e^{-\frac{r}{2}} - e^{-\frac{2r}{3}}\bigr]e^{ry} - e^{rx} + e^{-r} e^{r(x+y)}\right)^2}, & 0 < x < \frac{1}{3} \text{ and } \frac{2}{3} < y < 1,
\\
\frac{r e^{r(x + y)} \bigl(1 - e^{-\frac{r}{6}} + e^{-\frac{r}{3}}\bigr) \bigl(1-e^{-\frac{2r}{3}}\bigr)}{\left(1 - e^{ry} + \bigl[e^{-\frac{r}{6}} -1 - e^{-\frac{r}{3}}\bigr]e^{rx} + \bigl[e^{-\frac{2r}{3}} - e^{-\frac{5r}{6}} + e^{-r}\bigr]e^{r(x+y)}\right)^2}, & \frac{1}{3} < x < \frac{2}{3} \text{ and } 0 < y < \frac{2}{3},
\\
 0, & \frac{1}{3} < x < \frac{2}{3} \text{ and } \frac{2}{3} < y < 1,
\\
\frac{re^{r(x + y)} \bigl(1-e^{-\frac{2r}{3}} \bigr) \bigl(e^{-\frac{r}{6}} + e^{-\frac{2r}{3}}\bigr) }{
\left(e^{-\frac{r}{6}} - e^{-\frac{r}{6}}e^{ry} +\bigl[- e^{-\frac{r}{2}}+e^{-\frac{2r}{3}} - 1\bigr]e^{rx} +  e^{-\frac{7r}{6}}e^{r(x+y)} \right)^2}, & \frac{2}{3} < x < 1 \text{ and } 0 < y < \frac{2}{3},
\\
\frac{r e^{r(x + y)}e^{-\frac{4r}{3}}\bigl(1- e^{-\frac{r}{3}}\bigr)  \bigl(1 + e^{-\frac{r}{2}}\bigr)}{
\left(\bigl[1 - e^{-\frac{r}{6}} + e^{-\frac{r}{2}}\bigr] - e^{-\frac{2r}{3}} e^{ry} - e^{-\frac{2r}{3}} e^{rx}+e^{-\frac{5r}{3}} e^{r(x+y)}\right)^2}, & \frac{2}{3} < x < 1 \text{ and } \frac{2}{3} < y < 1.
\end{cases}
\end{equation}
%Verified against Maple and Matlab files.
A plot for two different values of $r$ is in Figure \ref{Figure_32_nonconvex}.

\begin{figure}[t]
\begin{center}
   \includegraphics[width=0.48\linewidth]{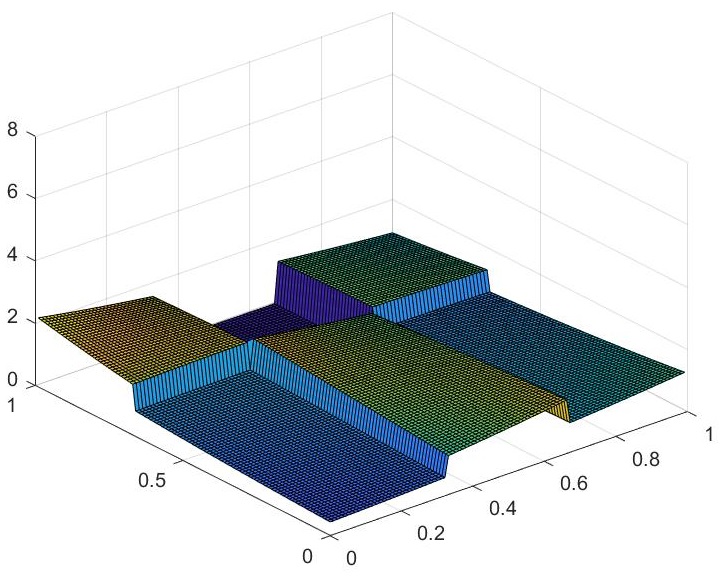} \hfill \includegraphics[width=0.48\linewidth]{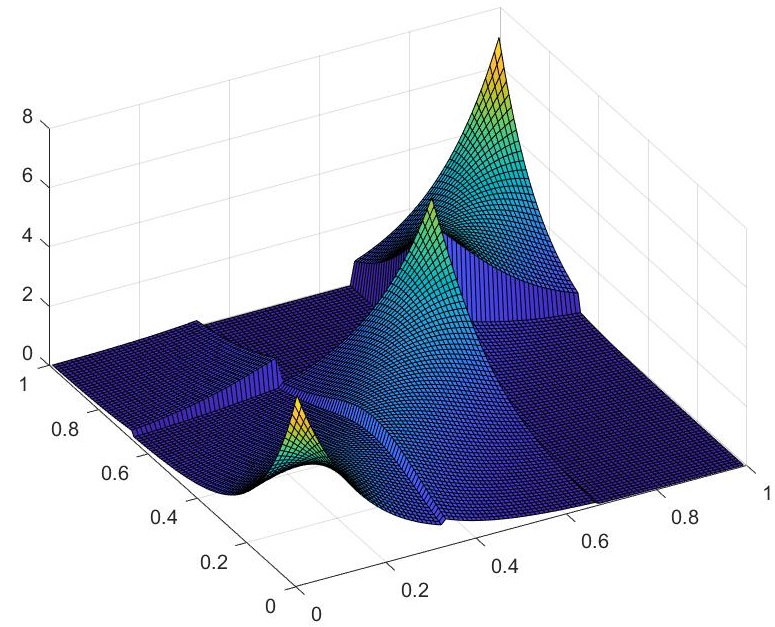}
\end{center}
        \caption{\label{Figure_32_nonconvex}  Solution of EL equations for $I=\protect\begin{pmatrix} 1 &0 & 1\\ 1 & 1 & 1\protect\end{pmatrix}$, $\mathrm x=(0,\tfrac{1}{3},\tfrac{2}{3},1)$, $\mathrm y=(0,\tfrac{2}{3},1)$ via \eqref{eq_x51}. Left panel: $r=-1$. Right panel: $r=7$.}
\end{figure}

\subsection{Non-rectangular example} Our setting, explained in Section \ref{Section_asymptotic_regime}, admits a natural generalization to a larger class of domains than just rectilinear polygonal domains. We can deal with domains $\Omega$, whose complexity grows with $N$ by allowing $k$ and $\ell$ to grow, and in the limit $\Omega^{X,Y,I}$ can approximate essentially arbitrary subsets of $[0,1]^2$. Theorems \ref{Theorem_LDP} and \ref{Theorem_maximizer_small_r} extend to such setting with no changes. Theorem \ref{Theorem_convex_solution} is no longer available, because the number and complexity of equations \eqref{eq_poly_1}, \eqref{eq_poly_2}, \eqref{eq_poly_3} grows with $k$ and $\ell$. However, in some situations at least in the $r=0$ case we can still find an explicit formula for the maximizer of the permuton energy. Let us provide an example.

Consider the domain $U$
obtained from $[0,1]^2$ by removing the region below the line $ax+by=1$ for $a,b>1$,  see Figure \ref{Figure_triangle}. Set $r=0$ and assume that $\Omega^{X,Y,I}$ approximates $U$ as $N\to\infty$. Then a generalization of Corollary \ref{Corollary_limit_shape}, yields that the permuton $g(x,y)$ corresponding to the limit shape should be a maximizer of
\begin{equation}
\label{eq_x52}
-\int_0^1\int_0^1 g(x,y) \ln g(x,y)\, \dd x\, \dd y,
\end{equation}
subject to the conditions
\begin{itemize}
 \item $g(x,y)=0$ for $ax+ by<1$,
 \item $\int_0^1g(x,y) \dd x=1$, $0\le y\le 1$, and $\int_0^1 g(x,y)\dd y=1$, $0\le x\le 1$.
\end{itemize}
In the $r=0$ case the Euler-Lagrange equations \eqref{eq_rectangle_condition} mean a factorization
$$
 g(x,y)=\lambda(x) \mu(y) \mathbf 1_{ax+by\ge 1}.
$$
Let us find $\lambda$ and $\mu$ for which $g$ has uniform marginals. We can assume that the function $\lambda(x)$ has some constant value $c_1$
for $x>1/a$, and the function $\mu(y)$ has some constant value $c_2$  for $y>1/b$. The assumption is based on the fact that the distribution of the random restricted permutations of interest is invariant under left and right multiplications with permutations of $\{N/a,\dots,N\}$ and $\{N/b,\dots,N\}$, respectively. Hence, the limit shape should also have this invariance. The uniform
marginals conditions become:
\begin{align}
\label{la}\lambda(x)c_2\, \left(1-\frac1b\right)+\lambda(x)\int_{(1-ax)/b}^{1/b}\mu(y)\,dy&=1,& \text{for $0<x<\frac1a$},\\
\label{mu}\mu(y)c_1\, \left(1-\frac1a\right)+\mu(y)\int_{(1-by)/a}^{1/a}\lambda(x)\,dx&=1,&\text{for $0<y<\frac1b$},\\
\label{la2}c_1\int_{0}^{1/b}\mu(y)\,dy+c_1c_2\, \left(1-\frac1b\right)&=1,& \text{for $\frac1a<x<1$},\\
\label{mu2}c_2\int_{0}^{1/a}\lambda(x)\,dx+c_1c_2\, \left(1-\frac1a\right)&=1,& \text{for $\frac1b<y<1$}.
\end{align}

We can solve this system as follows. Divide \eqref{la} by $\lambda(x)$ and differentiate to get
\be\label{f2}\frac{\lambda'(x)}{\lambda(x)^2}=-\frac{a}{b}\,\mu\left(\frac{1-ax}{b}\right).\ee
Likewise from \eqref{mu} we get
\be\label{g2}\frac{\mu'(y)}{\mu(y)^2}=-\frac{b}{a}\,\lambda\left(\frac{1-by}{a}\right).\ee
Differentiating \eqref{f2} and plugging into \eqref{g2} yields a differential equation for $\lambda$:
$$\left(\frac{\lambda'(x)}{\lambda(x)^2}\right)' = -\frac{b}{a}\, \frac{\lambda'(x)^2}{\lambda(x)^3}.$$
Under the substitution $\lambda(x)=e^{F(x)}$ this becomes
$$F''(x)=\left(1-\frac{b}{a}\right)F'(x)^2,$$
giving
$$F(x) = c_4-\frac{a}{a-b}\log((a-b)x+c_3)$$ for some constants $c_3,c_4$.
A similar expression can be obtained for $\mu(y)$ with $a,b$ reversed.
Plugging back into \eqref{la2} and \eqref{mu2} to compute the constants yields the final answer:
$$g(x,y)  = \begin{cases}
 b (b-1)^{\frac{b}{a-b}} ((a-b)x+b-1)^{\frac{a}{b-a}}, &
   x<\frac{1}{a}\text{ and }y>\frac{1}{b},
   \\
 a (a-1)^{\frac{a}{b-a}}
  ((b-a)y+a-1)^{\frac{b}{a-b}}, &
   x>\frac{1}{a}\text{ and } y<\frac{1}{b},
   \\
 b^{\frac{a}{a-b}}a^{\frac{b}{b-a}} ((a-b)x+b-1)^{\frac{a}{b-a}}
   ((b-a)y+a-1)^{\frac{b}{a-b}}, &
   a x+b y>1\text{ and } x<\frac{1}{a}\text{ and }
   y<\frac{1}{b}, \\
 (\frac{b}{b-1})^{\frac{b}{b-a}}(\frac{a}{a-1})^{\frac{a}{a-b}}, &
   x>\frac{1}{a}\text{ and } y>\frac{1}{b},\\
   0,&ax+by<1.
\end{cases}$$
Figure \ref{Figure_triangle} gives an example. The reader might enjoy taking the limit $b\to a$ of the expression for $g(x,y)$.
\begin{figure}[t]
\begin{center}
   \includegraphics[width=0.7\linewidth]{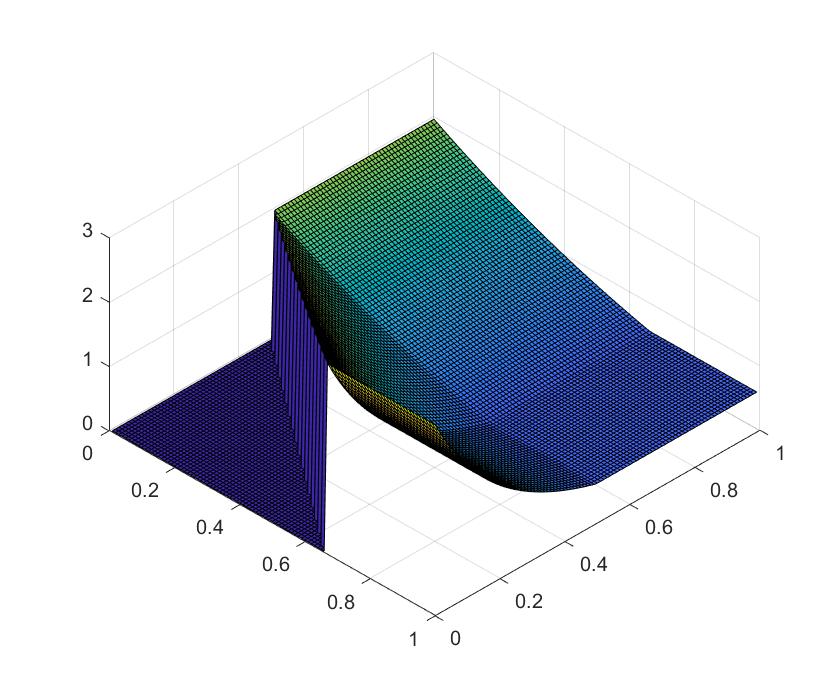}
\end{center}
\caption{\label{Figure_triangle} Maximizing permuton $g(x,y)$ for $r=0$ and the region $\frac32x+2y>1.$}
\end{figure}

\bibliographystyle{alpha}
\bibliography{6v_perm}

\end{document}